\documentclass[reqno]{amsart}
\usepackage{hyperref} 
\usepackage{graphicx}
\usepackage{tikz}
\usepackage{amsthm,amsmath,verbatim,amssymb}
\usepackage{color}
\usepackage{arydshln}

\theoremstyle{plain}
\newtheorem{lem}{Lemma}

\newtheorem{defn}{Definition}
\newtheorem{prop}{Proposition}
\newtheorem{thm}{Theorem}
\newtheorem{cor}{Corollary}

\newcommand{\N}{\mathbb{N}}
\newcommand{\abs}[1]{\left\vert#1\right\vert}
\newcommand{\norm}[1]{\left\Vert#1\right\Vert}
\newcommand{\1}{\mathbf{1}}
\newcommand{\ep}{\epsilon}
\renewcommand{\i}{\mathrm{i}}

\newcommand{\T}{{\bf{T}}}
\renewcommand{\hat}{\widehat}
\newcommand{\x}{{\bf{x}}}
\newcommand{\y}{{\bf{y}}}

\newcommand{\E}{\mathbb{E}}
 
\renewcommand{\d}{\textrm{d}}

\newcommand{\beq}{\begin{equation}}
	\newcommand{\eeq}{\end{equation}}
\newcommand{\beqq}{\begin{eqnarray*}}
	\newcommand{\eeqq}{\end{eqnarray*}}
\newcommand{\bc}{\begin{center}}
	\newcommand{\ec}{\end{center}}

\newcommand{\de}{determinantal point }  
\begin{document}
	
	\title[Multivariate linear statistics]{Determinantal point processes on  spheres: multivariate linear statistics}
	\author[Feng]{Renjie Feng}
	\author[G\"otze]{Friedrich G\"{o}tze}
	\author[Yao]{Dong Yao}
	
	\address{Faculty of Mathematics, Bielefeld University, Germany.}
	\email{rfeng@math.uni-bielefeld.de}
	\address{Faculty of Mathematics, Bielefeld University, Germany.}
	
	\email{goetze@math.uni-bielefeld.de}
	\address{Research Institute of Mathematics, Jiangsu Normal University, China. }
	\email{dongyao@jsnu.edu.cn}

	\date{\today}
	\maketitle
	\centerline{\it In memory of Steve Zelditch (1953-2022)}
	\begin{abstract}
		In this paper, we will derive the first and 2nd order Wiener chaos decomposition for the multivariate  linear statistics of   the \de processes associated with the spectral projection kernels on the unit spheres $S^d$.  We will first get a graphical representation for the cumulants of multivariate linear statistics for any \de process.  The main results then follow from the very precise estimates and identities regarding the spectral projection kernels and the symmetry of the spheres. 
	\end{abstract}

	\section{Introduction}
	
	
	The determinantal point process is an important class of point processes with applications in  random matrix theory, statistical mechanics, quantum mechanics, etc.  It's also called the Slater determinant in quantum mechanics that is to describe the wave function of a multi-fermionic system.  
	In this paper,  we will consider determinantal point processes on the unit spheres associated with  the spectral projection kernels of the Laplace operator with respect to the standard round metric. Such spectral projection kernels can be represented in terms of the spherical harmonics, which are one of the most fundamental wave functions in quantum mechanics to describe particles confined to the spheres.

	Let  $\Phi$  be a point process sampled on the space $\mathcal X$.  The $k$-th joint intensity function 	$\rho_k$ of the point process $\Phi$ 
	is  defined by
	\begin{equation}\label{defrho1}
		\E\Big[\sum_{(x_1,\ldots, x_k)\in \Phi_*^k }f(x_1,\ldots, x_k)\Big]=\int_{\mathcal X^k}
		f(x_1,\ldots, x_k)\rho_k(x_1,\ldots, x_k)dx_1\ldots dx_k,
	\end{equation}
	where $f$ is any bounded measurable function and   the set 
	\begin{equation}\label{def of *}
		\Phi^k_*:=\{(x_1,\ldots, x_k): x_i\in\Phi, \, \forall 1\leq i\neq j\leq k,
		x_i\neq x_j\}.
	\end{equation}
	If $\Phi$  is a determinantal point process associated with some kernel function $K$, then its $k$-th joint intensity function reads 
	\begin{equation}\label{rhodet}
		\rho_k(x_1, \ldots, x_k)=\det\Big(K(x_i,x_j)_{1\leq i\leq j\leq k}\Big),
	\end{equation}
	where $K(x_i,x_j)_{1\leq i,j\leq k}$ is  a $k\times k$ matrix whose $(i,j)$ entry is $K(x_i,x_j)$.  
	
	In this paper we will focus on the case when $K$ is the spectral projection kernel on the unit sphere $S^d$ with $d\geq 2$, defined as follows. 
	The Laplace operator  on  $S^d$ with respect to the standard round metric   has discrete spectrum $\Big\{\lambda_n=-n(n+d-1), n=0, 1, 2, ....\Big\}$. Here,  the round metric is the pullback of the Euclidean metric under the inclusion map $i: S^d\to\mathbb R^{d+1}$. For a given eigenvalue $\lambda_n$,   the corresponding eigenfunctions 
	are called the spherical harmonics of level $n$. Let $\mathcal H_n(S^d)$ be the space of the spherical harmonics of level $n$. Then one has  \cite{AH}
	\beq\label{knquan}k_n:=\text{dim} \mathcal H_n(S^d)=\frac{2n+d-1}{n+d-1} 
	{\begin{pmatrix}
			n+d-1\\
			d-1\\
	\end{pmatrix}},\eeq 
	which admits the asymptotic estimate (by $d\geq 2$)
	\beq\label{asykn}k_n \sim 2n^{d-1}/\Gamma(d).\eeq 	
	Let  $K_n$ be the spectral projection\beq \label{pmap}K_n: L^2(S^d) \to \mathcal H_n(S^d),\end{equation} and we denote by $K_n(x,y)$ its kernel. 

Now we define a determinantal point process $\Phi_n$ on $S^d$ associated with the  kernel $K_n(x,y)$.  Here the total number of points in $\Phi_n$ is almost surely $k_{n}$. 
Note that 
$\Phi_n$ can be alternatively defined by sampling $k_{n}$ points on $S^d$ according to the probability density
\beq
\frac{1}{k_{n}!} \det\Big(K_{n}(x_i,x_j)_{1\leq i,j\leq k_n}\Big).
\eeq
Given a function $f(x_1,..,x_k)$ of $k\geq 1$ variables, we define the multivariate linear statistics  	\beq\label{deflnf0} L_nf: =\sum_{(x_1,\ldots, x_k)\in \Phi_{n,*}^k }f(x_1,\ldots, x_k),\eeq
where 
\beq
\Phi^k_{n,*}:=\{(x_1,\ldots, x_k): x_i \in \Phi_n^{}, \,x_i\neq x_j,\, \forall 1\leq i\neq j\leq k \}.
\eeq
Multivariate linear statistics of this form are usually called U-statistics. 

For $1\leq i\leq k$, we define the \emph{$i$-margin} function $f_i$ by integrating $f$ with respect to all variables over $S^d$ except $x_i$, i.e., 
\begin{equation}\label{def_fi}
	f_i(x)=\int_{(S^d)^{ k-1}} f(x_1,\ldots, x_{i-1},x,x_{i+1},\ldots, x_k)dx_1\cdots dx_{i-1}dx_{i+1}\cdots dx_k. 
\end{equation} Here,  we denote by $dx$ the volume element with respect to the standard round metric on $S^d$. If $k=1$, the 1-margin function is defined to be $f$ itself.  

For $1\leq i<j\leq  k$,  we define the \emph{$(i,j)$-margin} function $f_{i,j}$ to be
\begin{equation}\label{ijmargin}
	f_{i,j}=\int_{(S^d)^{ k-2}} f(x_1,\ldots, x_k)dx_1\cdots dx_{i-1}dx_{i+1}\cdots dx_{j-1}dx_{j+1}\cdots dx_k. 
\end{equation}
In this article, we will study the limiting distribution of the multivariate linear statistics $L_nf$.  We first give an asymptotic expansion for the expectation of $L_nf$. 	
\begin{thm}\label{mean}Let $f(x_1,.., x_k)$ be a bounded function of $k$ variables. 
	We have  		\begin{equation}\label{eq:mean}
		\begin{split}
			\E(L_nf)=&\left(\frac{k_n}{s_d}\right)^k \int_{(S^d)^{ k}} f(x_1,\ldots,x_k)dx_1\cdots dx_k\\
			&- \frac{k_n^{k-1}}{s_d^k}\sum_{1\leq i<j\leq k}
			\frac{2^{d-1}}{\Gamma(d)\pi} \Gamma\Big(\frac{d}{2}\Big)^2 \int_{S^d} \int_{S^d}  \frac{f_{i,j}(x,y)}{\sin^{d-1} (\arccos (x\cdot y)) } dxdy\\&+o(n^{(d-1)(k-1)}),
		\end{split}
	\end{equation}
	where $s_d=2\pi^{\frac{d+1}2}/\Gamma(\frac{d+1}2)$ is the surface area of $S^d$. 
\end{thm}

By estimating the growth order of the cumulants of $L_nf$, we can prove the following central limit theorem for  $L_n f$.  
\begin{thm}\label{clt}
	Let $f$ be a bounded function of $k$ variables on $S^d$. 
	Assume that
	\begin{equation}\label{def_F}
		F(x):=\sum_{i=1}^k f_i(x)
	\end{equation}
	is not constant almost everywhere in $x\in S^d$, 
	then it holds that 
	\begin{equation}\label{q2case1}
		\lim_{n\to\infty}\frac{1}{k_n^{2k-1}} \mathrm{Var}(L_nf)\\
		=\frac{2^{d-2}}{s_d^{2k}\Gamma(d)\pi} \Gamma\Big(\frac{d}{2}\Big)^2 \int_{S^d} \int_{S^d}  \frac{(F(x)-F(y))^2}{\sin^{d-1} (\arccos (x\cdot y)) }dxdy>0.
	\end{equation}
	In addition,  $L_nf$ is asymptotically normal, i.e.,
	\beq \label{connormal}
	\frac{L_nf-\E( L_nf)}{(\mathrm{\mathrm{Var}}(L_nf))^{\frac{1}{2}}} 	\xrightarrow{\textrm{d}}N(0,1),
	\eeq
	where $N(0,1)$ is the standard Gaussian distribution and the notation $	\xrightarrow{\textrm{d}}$ means the convergence in distribution. 
\end{thm}
Combining Theorem \ref{mean} and Theorem \ref{clt}, we have the following corollary.
\begin{cor}\label{cor:1}
	Under the assumption of Theorem \ref{clt}, 
	$$\left(L_n f -   \left(\frac{k_n}{s_d}\right)^k \int_{(S^d)^{k}} f(x_1,\ldots,x_k)dx_1\cdots dx_k \right)\operatorname{Var}(L_nf)^{-1/2}   
	\xrightarrow{\textrm{d}}N(0,1).
	$$
\end{cor}
When the assumption of Theorem \ref{clt}  fails, i.e., $F(x)$ is constant almost everywhere,  
the right hand side of \eqref{q2case1} will be degenerate, i.e., 
$\mathrm{\mathrm{Var}}(L_nf)$ will have strictly smaller growth order than $k_n^{2k-1}=\Theta(n^{(d-1)(2k-1)})$. 
For such degenerate case, our next theorem shows that for a class of test functions, the limiting distribution is given by a mixture of centered chi-square distributions, i.e., the 2nd order Wiener Chaos. 

We now consider the following two invariance conditions on the bounded test function $f(x_1,\ldots, x_k)$, $k\geq 2$. 
\begin{itemize}
	\item $f$ is invariant under permutations, i.e.,
	\begin{equation}\label{sym}
		f(x_1,\ldots, x_k)=f(x_{\sigma(1)},\ldots, x_{\sigma(k)}),  \forall  \, \sigma \in \text{Sym}(k).
	\end{equation}
	\item 
	We assume that  the $(1,2)$-margin function   $f_{1,2}(x_1,x_2)$  only depends on their spherical distance dist($x_1,x_2$) (abbreviated as  $\d(x_1,x_2)$), i.e., 
	\begin{equation}\label{dis}
		f_{1,2}(x_1,x_2)=f_{1,2}(x_1',x_2'),\,\,\, \forall \, \d(x_1,x_2)=\d(x_1',x_2').
	\end{equation}
	
\end{itemize}
We will show that if the test function $f$ satisfies these two assumptions, then $F(x)$ must be   constant on the sphere, and thus the variance will be degenerate. 

As a remark, the condition \eqref{sym} is not an essential one. We can always symmetrize a function $f$  by considering the average $$ 
\bar{f}(x_1, \ldots, x_k)= \frac{1}{k!}\sum_{\sigma\in \textrm{Sym}(k)}f(x_{\sigma(1)},\ldots, x_{\sigma(k)}),
$$
and this will yield $L_nf=L_n \bar{f}$ by \eqref{deflnf0}. 

There is an important class of  test functions that satisfy these two assumptions.  For example, given $\delta>0$,  if we choose   \beq f(x_1, x_2)= \1[\d(x_1, x_2)<\delta],\eeq 
where  
the indicator function  is equal to 1 if the distance $\d(x_1, x_2)<\delta$ and 0 otherwise, then the random variable $L_nf$ will be the number of pairs of random points whose distances are less than $\delta$. Similarly, if we take \beq f(x_1, x_2, x_3)=\1[\d(x_1, x_2)<\delta, \,\d(x_1, x_3)<\delta, \,\d(x_2, x_3)<\delta],\eeq then $L_nf$ will count the number of triangles where the three vertices of the triangle are within distance $\delta$. These types of counting statistics are useful tools to study the topology of random complexes built over random point processes, due to its connections with Betti numbers,  e.g., \cite{BO, KM, YA}. 
Our main result Theorem \ref{clt2} below implies that such types of
counting statistics of the \de process on $S^d$ converge to the 2nd order Wiener chaos.

Under conditions \eqref{sym} and \eqref{dis}  we can determine the growth order of $\mathrm{\mathrm{Var}}(L_nf)$ and find the limiting distribution of $L_nf$.	 We define the function
\begin{equation}\label{hath}
	\hat{h}(x,y):=\int_{S^d} (f_{1,2}(x,y)-f_{1,2}(x,z))\sin^{-(d-1)}(\arccos(z \cdot y) )dz.
\end{equation}
We will see that $\hat{h}$ is a bounded symmetric function, and thus we can consider it as a Hilbert-Schmidt integral operator acting on $L^2(S^d)$. Then this operator is compact and self-adjoint. Therefore we have the spectral decomposition 
\begin{equation}\label{specdecomp}
	\hat{h}(x,y)=\sum_{j=1}^{\infty} z_j w_j(x)w_j(y), 
\end{equation}
where $\{z_j,j\geq 1\}$ are eigenvalues of the operator, and $\{w_j,j\geq 1\}$ are the corresponding eigenfunctions which form an orthonormal basis of $L^2(S^d)$. 

The following theorem  states that the multivariate linear statistics will tend to a mixture of centered chi-squared distributions in the degenerate case. 
\begin{thm}\label{clt2}
	For any bounded function $f(x_1,\ldots,x_k)$ with  $k\geq 2$ satisfying conditions \eqref{sym} and \eqref{dis}, we have
	\begin{equation}\label{q2case2}
		\lim_{n\to\infty} \frac{\mathrm{Var}(L_nf)}{k_n^{2k-2}}=	\frac{2C_{d}^2 k^2(k-1)^2}{\Gamma (d)^2s_d^{2k}} \int_{(S^d)^2} \hat{h}(x,y)^2dxdy, 
	\end{equation}
	where the constant $	C_{d}:=\frac{2^{d-2} \Gamma(d/2)^2}{\pi}.$ 
	Furthermore,  we have 
	\begin{equation}\label{limcase2}
		\Big(L_nf-\E(L_nf)\Big) \left( \frac{k_n}{s_d}\right)^{-k}  \left(\frac{C_{d} k(k-1)}{n^{d-1}}\right)^{-1} 
		\xrightarrow{\textrm{d}}
		\sum_{i=1}^{\infty} z_i (\chi_i-1)/2,  
	\end{equation}
	where $\chi_i,i\geq 1$ are independent chi-squared random variables with one degree of freedom and  $	\sum_{i=1}^{\infty} z_i (\chi_i-1)/2$ is understood as the $L^2$-limit of  $	\sum_{i=1}^N z_i (\chi_i-1)/2$ as $N\to\infty$.
\end{thm}
Similar to Corollary \ref{cor:1}, using the fact that $k_n \sim 2n^{d-1}/\Gamma(d)$, and Theorems \ref{mean} and \ref{clt2}, we deduce the following result. 
\begin{cor}
	Under the assumptions of Theorem \ref{clt2}, we have 	
	\begin{equation}
		\begin{split}
			&\left( \frac{k_n}{s_d}\right)^{-k}  \left(\frac{C_{d} k(k-1)}{n^{d-1}}\right)^{-1}	\left(L_nf-
			\left(\frac{k_n}{s_d}\right)^k \int_{(S^d)^{ k}} f(x_1,\ldots,x_k)dx_1\cdots dx_k \right.\\
			&	\left.	+ \frac{k_n^{k-1}}{s_d^k}\sum_{1\leq i<j\leq k}
			\frac{2^{d-1}}{\Gamma(d)\pi} \Gamma\Big(\frac{d}{2}\Big)^2 \int_{S^d} \int_{S^d}  \frac{f_{i,j}(x,y)}{\sin^{d-1} (\arccos (x\cdot y)) } dxdy
			\right)  \\
			&	\xrightarrow{\textrm{d}} 
			\sum_{i=1}^{\infty} z_i (\chi_i-1)/2. 
		\end{split}
	\end{equation}
\end{cor}
Note that the limiting distribution can be rewritten in the form of the 2nd order Wiener chaos
\beq \label{hermite} 
\sum_{i=1}^{\infty} z_i H_2(X_i)/2,
\eeq
where $H_2(x)=x^2-1$ is the Hermite polynomial of degree 2, and $X_i$ are independent and identically distributed (i.i.d.) standard Gaussian random variables $N(0,1)$.


There is a vast literature on the univariate linear statistics of \de processes, e.g., \cite{TY, AS2, AS}. 
There are also  very few works that give conditions for a Gaussian limit of multivariate linear statistics, e.g., \cite{BYY}.
But to the best of our knowledge,   Theorem \ref{clt2} is the very first result on the multivariate linear statistics for  \de processes beyond the Gaussian limit case. 


Theorem \ref{clt} and Theorem \ref{clt2} are proved  by the method of cumulants. We will first derive a graphical representation for the cumulants of the multivariate linear statistics for any \de process in Lemma \ref{soss}, which generalizes the well-known  formula for the univariate case (see \eqref{cummufor1} below). This graphical representation allows us to study the asymptotic properties of the cumulants by the  off-diagonal decay of the spectral projection kernel, where we have to prove Lemma \ref{productcontrol} and Lemma \ref{keylemma} bounding multiple integrals over the product of kernels. Exact identities and asymptotic expansions of the spectral projection kernels combined with the symmetry of the underlying space of the sphere  are two crucial ingredients for our proofs. For example, we repeatedly use the facts that the spectral projection kernel is  constant on the diagonal, and it satisfies very precise off-diagonal estimates  for all length scales, e.g., \eqref{hilbs}; the important fact that the integral operator $\hat h(x, y)$ defined in \eqref{hath} is symmetric is partially due to the symmetry of the sphere, etc. 

Contrary to the i.i.d. point process,  the \de process has the negative association property.   But our main results Theorem \ref{clt} and Theorem \ref{clt2} are still analogs of the classical Wiener chaos decomposition in the theory of U-statistics for i.i.d   random variables. 

Given i.i.d. random variables $X_{1}, \cdots, X_{n} $, 
Hoeffding's form for  U-statistics is the following (normalized) multivariate linear statistics, 
$$ 
U_{n}^{k}(g) ={n\choose k}^{-1} \sum_{1 \leq  i_{1} < \dots <i_{k} \leq  n }g (X_{i_{1}  }, \dots ,X_{i_{k}  } ),
$$  where $ g$ 
is a symmetric real-valued function of  $k$
variables. 

Without loss of generality, we assume $\mathbb E(g(X_1,.., X_k))=0$. 
Then Hoeffding in 1948 proved that,  if the variance $\text{Var}(g(X_1,.., X_k))<\infty$, then the following central limit theorem holds (Corollary 11.5 in \cite{JS}),
\beq \label{centld}n^{1/2}U^k_n(g)\xrightarrow{\textrm{d}}  N(0, k^2\delta_1).\end{equation}
Here, the constant $\delta_1$ is the variance  $$\delta_1=\text{Var}(g_1(X_1)),$$
where $$g_1(x):=\mathbb E(g(x, X_2,.., X_k)).$$
If the variance $\delta_1$ vanishes, that is the  limit of U-statistics for i.i.d. random variables is degenerate, then a  $\chi^2$-limit theorem holds for the rescaled statistics. To be  more precise, 
we suppose that 
$g_1(x)=\mathbb Eg(x, X_2,.., X_k)=0$  and $\mathbb E g^2(X_1,.., X_k)<\infty$,  then  we have (Corollary 11.5 in \cite{JS}),
\beq \label{xhic}nU^k_n(g) \xrightarrow{\textrm{d}} {k\choose 2} \sum_{i=1}^\infty \lambda_iH_2(Y_i), \eeq
where $H_2(x)=x^2-1$ is the Hermite polynomials of degree 2, $Y_i$ are i.i.d.  standard Gaussian random variables, and $\lambda_i$ are eigenvalues of the integral operator $A$ defined as follows.
Let  $d\mu$ be the probability density of the random variable $X_1$ and  set $$g_2(x, y):=\mathbb Eg(x, y, X_3,.., X_k).$$ 
For any bounded measurable function $f$, 
the operator $A$ is define by
\beq \label{aope}(A f)(y)=\int g_2(x, y) f(x) d\mu(x).\end{equation}
The formats of  results \eqref{centld} and \eqref{xhic}  are almost identical  to  Theorem \ref{clt} and  Theorem \ref{clt2}, respectively. The roles of $g_1(x_1)$ and $g_2(x_1, x_2)$ are replaced by  the $i$-margin function $f_i(x)$ and the $(i,j)$-margin function $f_{i,j}(x,y)$  respectively; when the variance vanishes, both the limiting distributions are the linear eigenvalue combination of   $H_2(Y_i)$, where the role of the symmetric integral operator $A$ is replaced by $\hat h(x,y)$.



In general, $U_n^k$ may exhibit the convergence in distribution  to the Wiener chaos with arbitrary order (Theorem 11.3 in \cite{JS}).  For example, for the primitive completely degenerate case where 
$$
g\left(x_1, \ldots, x_k\right)=\prod_{i=1}^k \mathfrak g\left(x_i\right)
$$
with $\mathbb{E} \mathfrak g \left(X_1\right)=0$ and $\mathbb{E} \mathfrak g^2\left(X_1\right)=\sigma^2<\infty$, one has the   convergence
\beq\label{hermk}
\frac{n^{k / 2} U_n^k\left(g\right)}{\sigma^k} \xrightarrow{\textrm{d}}  H_k(Y),
\eeq
where $H_k(x)$ is the Hermite polynomial of degree $k$ and $Y$ is the standard Gaussian random variable. 


Therefore,  we may expect that the multivariate linear statistics of the \de process associated with the spectral projection kernel on $S^d$ also admits some  kind of Wiener chaos decomposition. Actually, our method, especially the representation formula in Lemma  \ref{soss}, can be applied to any other determinantal point process such as  CUE, GUE,  the complex Ginibre ensemble  in random matrix theory and Gaussian analytic functions in random polynomial theory. And the similar results may hold as well, but note that one has to change the conditions especially \eqref{dis} for the test functions  to others according to the symmetry and the invariance of the underlying space and the kernel.  

\bigskip

\emph{Notation.}  In this paper, we use $C$ (or $c$) to denote some constants independent of $n$, whose specific values may change from line to line.
For a sequence of numbers $a_n$ and $b_n$, we write $a_n=o(b_n)$ if 
$b_n\neq 0$ and $\lim_{n\to\infty}{a_n}/{b_n}= 0$;  $a_n=O(b_n)$ if there exists some constant $C$ such that $\abs{a_n}\leq C\abs{b_n}$; $a_n=\Theta(b_n)$ if  $a_n=O(b_n)$  and $b_n=O(a_n)$; $a_n\sim  b_n$ if  $\lim_{n\to\infty} a_n/b_n = 1$.


\section{A graphical representation of cumulants}\label{repformula}
In this section we will derive a graphical representation  of the cumulants for the multivariate linear statistics of any determinantal point process. 

Given a random variable $X$, its $m$-th cumulant $Q_m(X)$ is defined to be the coefficient in the formal expansion of $\log \mathbb{E} \exp(\i tX)$, 
\beq
\log \mathbb{E}\exp(\i tX)=\sum_{m=1}^{\infty} \frac{Q_m(X)}{m!}(\i t)^m.
\eeq
A  \emph{partition} of a set $S$
is an unordered collection $R=\{R_1,\ldots, R_{\ell}\}$  of nonempty subsets of $S$
where  $\ell$ is some positive integer not exceeding $\abs{S}$. In addition,  $R$ satisfies the following two conditions:
\begin{itemize}
\item  $R_i\cap R_j=\emptyset$ for $i\neq j$.
\item $\cup_{i=1}^\ell R_i=S$.
\end{itemize}

Let $m$ be any positive integer. We denote by $\Pi(m)$  the set of partitions of $\{1,2,\cdots, m\}$.
The moments of $X$ can be derived from its cumulants as follows,
\beq\label{moment_cummu}
\mathbb{E}(X^m)=\sum_{ R=\{R_1, \ldots, R_{\ell}\} \in \Pi(m)   }Q_{\abs{R_1}}\ldots Q_{\abs{R_{\ell}}}.
\eeq
On the other hand, the cumulants can be expressed by moments as 
\beq\label{cummu_moment}
Q_m(X)=\sum_{R=\{R_1,\ldots, R_{\ell}\} \in \Pi(m)} (-1)^{\ell-1}(\ell-1)!\Pi_{i=1}^{\ell} \mathbb{E}X^{\abs{R_i}}.
\eeq
Some simple properties of cumulants include $$Q_1(X)=\E(X),\,\, Q_2(X)=\mathrm{\text{Var}}(X), \,\,Q_m(cX)=c^m Q_m(X).$$ If $X$ is a Gaussian random variable, then  $Q_m(X)=0$ for all $m\geq 3$. 

Similarly to the method of moments, to show that  $X_n$ converges in distribution to $X$, it suffices to prove that the $m$-th cumulant of $X_n$ converges to $Q_m(X)$ for all fixed $m$ (as long as the limit is uniquely determined by its cumulants).  For the special case that $X$ is Gaussian distributed and $X_n$ has mean 0,
it suffices to prove
$$
\lim\limits_{n\to\infty} \frac{Q_m(X_n)}{\mathrm{\text{Var}}(X_n)^{\frac{m}{2}}}=0
$$  for all sufficiently large $m$ (\cite[Lemma 3]{AS}). 
Let $\Phi$ be a determinantal point process  on the space $\mathcal X$ associated with the kernel $K(x,y)$.  In the followings,  we will derive a formula for the cumulants of the
multivariate linear statistics. We will expand the $m$-th power of the multivariate linear statistics and express it in the form of  \eqref{moment_cummu}, then  the formula for the cumulants can be found directly from this expression. 

To expand $(\sum_{(x_1,\ldots, x_k)\in \Phi_{*}^k}f(x_1,..,x_k))^m$,  we have
$km$ points $x_1,\ldots, x_{mk}$ (counting multiplicities)
appearing in 
the product $f(x_1,\ldots, x_k)\cdots f(x_{mk-k+1},\ldots,x_{mk})$.  We write $y_{i,j}:=x_{(i-1)k+j}$ for $1\leq i\leq m, \,1\leq j\leq k$, and set $\y_i:=(y_{i,1},\ldots, y_{i,k})$.
Then  we have 	\begin{equation}
\left( \sum_{(x_1,\ldots, x_k)\in \Phi_{*}^k} f(x_1,\ldots, x_k) \right)^m
=\sum_{\y_1,\ldots, \y_m\in \Phi_*^k} f(\y_1)\cdots f(\y_m). 
\end{equation}
We first introduce a notation: given any positive integer $p$, we define the set $$[p]:=\big\{1,\ldots, p\big\}.$$
To find the relations among the points $x_1,\ldots, x_{mk}$, we define by $$M(m,k):=\text{Map}([m],{[km]^k_*})$$  the set of  all maps  from $[m] $ to 
\beq \label{inteset}
[km]^{ k}_*:=\Big\{(i_1,\ldots, i_k) \in [km]^{ k}: i_j\neq i_{\ell}, \,\,\forall \, 1\leq j< \ell \leq k\Big\}.
\eeq  	
To be more precise,  let $\T$ be an element in  $M(m,k)$, then we can rewrite it as 
$$\T:=(T_1,\ldots, T_m),$$
where each $T_i$ is the image of $i\in \{1,2,.., m\}$ under the map $\T$ and   $$T_i\in \Big\{(i_1,\ldots, i_k): i_j\in [km]\,\, \text{and}\,\, i_j\neq i_{\ell}, \,\,\, \forall \, 1\leq j< \ell \leq k\Big\}.$$
We also write $T_{i}=(T_{i,1},\ldots, T_{i,k})$ where $T_{i,j}$ is the $j$-th component of the $k$-tuple $T_i$.
For example, when $m=3$ and $k=2$, then $\T,\T',\T''$ defined as follows   all belong to  $M(3,2)$, 
\beq\label{defT}
T_1=(1,2), T_2=(1,4), T_3=(2,4).
\eeq
\beq\label{defT'}
T'_1=(1,3),T'_2=(1,6),T'_3=(3, 6).
\eeq
\beq\label{defT''}
T''_1=(1,2),T''_2=(1,4),T''_3=(5,6).
\eeq
We say two maps $\T, \hat{\T} \in  M(m,k)$ are \emph{equivalent} if they differ by a permutation of $[km]$, i.e. by composing with a permutation  they become the same map.  We denote by $S(m,k)$ the set of all equivalence classes of $M(m,k)$. 
As an example, the $\T$ and $\T'$ defined in \eqref{defT} and \eqref{defT'} are
equivalent since the permutation (23)(46) brings $\T$ to $\T'$. But $\T''$ defined in \eqref{defT''} is neither equivalent to $\T$ nor $\T'$. 

For any $\T \in   M(m,k)$,  we can construct  a graph for it, which we call \emph{$\T$-graph}. The $\T$-graph is constructed in two steps. Initially there 
are $mk$ vertices in total, indexed by $(i,j)$ for $1\leq i\leq m,1\leq j\leq k$. 
First for each $1\leq i\leq m$ and $1\leq j\leq k-1$, we draw a black edge between $T_{i,j}$ and $T_{i,j+1}$.
Then for any $(i,j)\neq (i',j')$ such that $T_{i,j}=T_{i',j'}$, we use a solid red edge to connect $(i,j)$ and $(i',j')$.
See Figure \ref{Example22} for the graphical representations of $\T$, $\T'$ and $\T''$.   One can see that
if $\T$ is equivalent to $\hat{\T}$, then  $\T$-graph is the same as  $\hat{\T}$-graph, and vice versa. Consequently, each equivalence class in $S(m,k)$ can be identified with a $\T$-graph.

%
%
%

\begin{figure}[h!]
\begin{center}
\begin{tikzpicture}
	\draw (0,0)--(2,0);
	\draw (-0.3,0.3)--(0.7,1.7);
	\draw (1.3, 1.7)--(2.3,0.3);
	\draw (0.5,2.0) node {$(1,1)$};
	\draw  (-0.3,-0.3) node {$(2,1)$};
	\draw  (-0.7, 0.3) node {$(1,2)$};
	\draw (1.5,2.0) node {$(3,2)$};
	\draw (2.3, -0.3) node {$(2,2)$};
	\draw (2.7, 0.3) node {$(3,1)$};
	\draw (0, 1.3) node {${T_1}$};
	\draw  (1.9, 1.3) node {${T_3}$};
	\draw (0.9,-0.4) node {${T_2}$};
	\draw[red] (0.7,1.7)--(0,0); 
	\draw[red] (-0.3,0.3)--(2.3,0.3); 
	\draw[red] (2.0,0)--(1.3,1.7);
	
	\draw (4.3,0)--(6.3,0);
	\draw (4,0.3)--(5,1.7);
	\draw (5.6, 1.7)--(6.6,0.3);
	\draw (4.8,2.0) node {$(1,1)$};
	\draw  (4,-0.3) node {$(2,1)$};
	\draw  (3.6, 0.3) node {$(1,2)$};
	\draw (5.8,2.0) node {$(3,2)$};
	\draw (6.6, -0.3) node {$(2,2)$};
	\draw (7, 0.3) node {$(3,1)$};
	\draw (4.3, 1.3) node {${T'_1}$};
	\draw  (6.2, 1.3) node {${T'_3}$};
	\draw (5.2,-0.4) node {${T'_2}$};
	\draw[red] (5,1.7)--(4.3,0); 
	\draw[red] (4,0.3)--(6.6,0.3); 
	\draw[red] (6.3,0)--(5.6,1.7);

	\draw (8.6,0)--(10.6,0);
	\draw (8.3,0.3)--(9.4,1.7);
	\draw (9.9, 1.7)--(10.9,0.3);
	\draw (9.1,2.0) node {$(1,1)$};
	\draw  (8.3,-0.3) node {$(2,1)$};
	\draw  (7.9, 0.3) node {$(1,2)$};
	\draw (10.1,2.0) node {$(3,2)$};
	\draw (10.9, -0.3) node {$(2,2)$};
	\draw (11.3, 0.3) node {$(3,1)$};
	\draw (8.6, 1.3) node {${T''_1}$};
	\draw  (10.5, 1.3) node {${T''_3}$};
	\draw (9.5,-0.4) node {${T''_2}$};
	\draw[red] (9.4,1.7)--(8.6,0);

\end{tikzpicture} 
\end{center}
\caption{Graphical view of $\T$ (left), $\T'$ (middle) and $\T''$ (right)}
\label{Example22}
\end{figure}
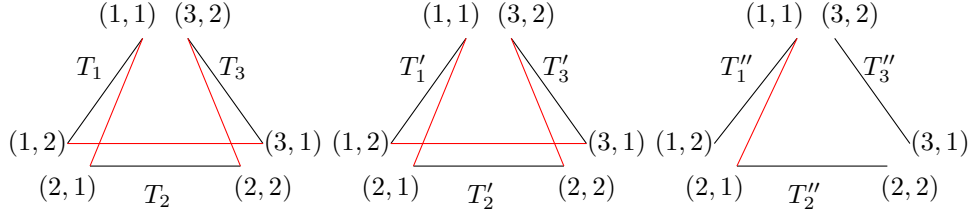


For $\T \in M(m,k)$, we define the size and the range of $\T$ as
$$\abs{\T}:=\abs{\cup_{i=1}^m T_i}, \,\,\,\text{Range}(\T)=\cup T_i.$$
For example, for the $\T$ defined in \eqref{defT} we have
$\abs{\T}=3$ and $\text{Range}(\T)=\{1,2,4\}.$

For notational simplicity, for a collection of indices ${\bf t}=(t_1,\ldots, t_k)$, we define $f(\textbf{t}):=f(x_{t_1},x_{t_2},\ldots, x_{t_k})$; by an abuse of nation, for $\T=(T_1,\ldots, T_m)$, we set $f(\T)=\Pi_{i=1}^m f(T_i)$; and we write $d\x$
as the volume element involved in the integration. 
By the definition of the determinantal point process, we have 
\begin{equation}\label{mom-permu}
\begin{split}
&\E\left[\left(\sum_{(x_1,\ldots,x_k) \in \Phi_{*}^k} f(x_1,\ldots, x_k)\right)^m\right]\\ =& \sum_{\T \in S(m,k)} \int_{\mathcal X^{\abs{\T}}} f(\T)\det\Big(K(x_i,x_j)_{ { i,j\in \text{Range}(\T)}}\Big)  d\x\\
=&\sum_{\T \in S(m,k)} \sum_{\sigma \in \text{Sym}(\text{Range}(\T))}  
\int_{\mathcal X^{\abs{\T}}} f(\T) {\text{sgn}(\sigma)}\Pi_{q\in \textrm{Range}(\T)}K(x_q,x_{\sigma(q)})   d\x.
\end{split}
\end{equation}
Here,  for any set $A$, $\text{Sym}(A)$ is the set of all permutations of the elements in $A$, and sgn($\sigma$) is the sign of the permutation $\sigma$.

For any $\T$ and $\sigma\in \text{Sym}(\text{Range}(\T))$
we can further construct a \emph{$(\T,\sigma)$-graph} $G$ by
adding dotted  red edges to the $\T$-graph. Specifically, for any $T_{i,j}\neq T_{i',j'}$, we add a dotted red edge between two vertices $(i,j)$ and $(i',j')$ if $\sigma(T_{i,j})=T_{i',j'}$ or $\sigma(T_{i',j'})=T_{i,j}$.  We say the pair $(\T,\sigma)$ is \emph{connected} if the $(\T,\sigma)$-graph is connected. 

For example, for  $\T$ defined in \eqref{defT},
the $(\T,\sigma)$-graph $G$ is connected for any $\sigma\in \text{Sym}(\{1,2,4\})$  because the $\T$-graph itself is already connected. 
On the other hand, for $\T''$ defined in   \eqref{defT''}, 
if $\sigma=id$ (the identity in the permutation group), then the $(\T'',\sigma)$-graph has two components.
However,  for $\sigma=(15)\in \text{Sym}(\{1,2,4, 5,6\})$  the $(\T'',\sigma)$-graph becomes connected.

\begin{figure}[h!] 
\begin{center}
\begin{tikzpicture}
	\draw (0,0)--(2,0);
	\draw (-0.3,0.3)--(0.7,1.7);
	\draw (1.3, 1.7)--(2.3,0.3);
	\draw (0.5,2.0) node {$(1,1)$};
	\draw  (-0.3,-0.3) node {$(2,1)$};
	\draw  (-0.7, 0.3) node {$(1,2)$};
	\draw (1.5,2.0) node {$(3,2)$};
	\draw (2.3, -0.3) node {$(2,2)$};
	\draw (2.7, 0.3) node {$(3,1)$};
	\draw (0, 1.3) node {${T_1}$};
	\draw  (1.9, 1.3) node {${T_3}$};
	\draw (0.9,-0.4) node {${T_2}$};
	\draw[red] (0.7,1.7)--(0,0); 
	\draw[red] (-0.3,0.3)--(2.3,0.3); 
	\draw[red] (2.0,0)--(1.3,1.7);
	
	\draw (4.3,0)--(6.3,0);
	\draw (4,0.3)--(5,1.7);
	\draw (5.6, 1.7)--(6.6,0.3);
	\draw (4.8,2.0) node {$(1,1)$};
	\draw  (4,-0.3) node {$(2,1)$};
	\draw  (3.6, 0.3) node {$(1,2)$};
	\draw (5.8,2.0) node {$(3,2)$};
	\draw (6.6, -0.3) node {$(2,2)$};
	\draw (7, 0.3) node {$(3,1)$};
	\draw (4.3, 1.3) node {${T''_1}$};
	\draw  (6.2, 1.3) node {${T''_3}$};
	\draw (5.2,-0.4) node {${T''_2}$};
	\draw[red] (5,1.7)--(4.3,0); 

	\draw (8.6,0)--(10.6,0);
	\draw (8.3,0.3)--(9.4,1.7);
	\draw (9.9, 1.7)--(10.9,0.3);	
	\draw (9.1,2.0) node {$(1,1)$};
	\draw  (8.3,-0.3) node {$(2,1)$};
	\draw  (7.9, 0.3) node {$(1,2)$};
	\draw (10.1,2.0) node {$(3,2)$};
	\draw (10.9, -0.3) node {$(2,2)$};
	\draw (11.3, 0.3) node {$(3,1)$};
	\draw (8.6, 1.3) node {${T''_1}$};
	\draw  (10.5, 1.3) node {${T''_3}$};
	\draw (9.5,-0.4) node {${T''_2}$};
	\draw[red] (9.4,1.7)--(8.6,0); 
	\draw[red, dotted, thick] (9.4, 1.7)--(10.9,0.3); 
	\draw[red, dotted, thick] (8.6,0)--(10.9,0.3);
\end{tikzpicture} 
\end{center}
\caption{$\T$  in \eqref{defT}, $\sigma=id$ (left);
$\T''$  in \eqref{defT''}, $\sigma=id$ (middle);
$\T''$ in \eqref{defT''}, $\sigma=(15)$ (right).
}
\label{Example2}
\end{figure}
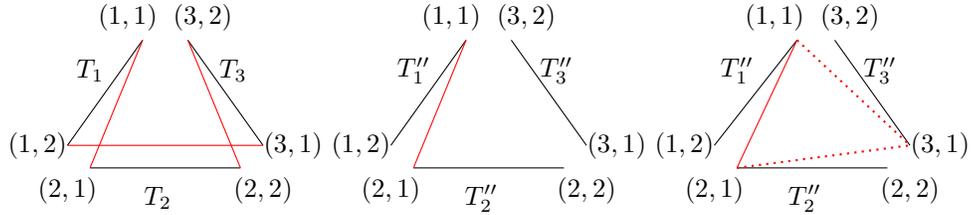

If a $(\T,\sigma)$-graph $G$ has $\ell$ connected components, then $G$  naturally induces a 
partition $R$ of $[m]$  into $\ell$ disjoint sets $\{R_1,\ldots, R_{\ell}\}$. 
For $1\leq j\leq \ell$, we set
$$
H_j:=\cup_{i\in R_j} T_i,  \,\, \sigma_j= \mbox{the restriction of } \sigma \mbox{ to } H_j.$$
Let $f(\T|_{R_j})=\prod_{i\in R_j}f(T_i)$.
 Then for the integral $$  
\int_{\mathcal X^{\abs{\T}}} f(T_1)f(T_2)\ldots f(T_m) {\text{sgn}(\sigma)}\Pi_{q\in \textrm{Range}(\T) }K(x_q,x_{\sigma(q)})   d\x, 
$$
we can split it into  a product of exactly $\ell$ integrals
$$
\prod_{j=1}^{\ell} \left( \int_{\mathcal X^{\abs{H_j}}} \textrm{sgn}(\sigma_j)
f(\T|_{R_j}) \prod_{q\in H_j}K_n(x_{q}, x_{\sigma_j(q)}) dx_{q}
\right).
$$
  For any integer-valued $r$, we define 	\begin{equation}\label{def*}\begin{split}
\mathcal{C}(r):=&\Big\{(\T,\sigma): \T\in S(r,k), \sigma \in \text{Sym}(\text{Range}(\T)),\\ &\, (\T,\sigma)\mbox{-graph is connected}
\Big\}.\end{split}
\end{equation}
The definition of $\{R_1,\ldots, R_{\ell}\}$ implies that, for each $1\leq j\leq \ell$, the pair $(\T|_{R_j}, \sigma_j)$ is in $\mathcal{C}(\abs{R_j}).$
Therefore, we have
\begin{equation}\label{findcumu}
\begin{split}
&\sum_{\T \in S(m,k)} \sum_{\sigma \in \text{Sym}(\text{Range}(\T))}  
\int_{\mathcal X^{\abs{\T}}} f(\T) {\text{sgn}(\sigma)}\prod_{q\in \textrm{Range}(\T)}K(x_q,x_{\sigma(q)})   d\x \\
=&\sum_{R=\{R_1,\ldots, R_{\ell}\} \in \Pi(m)}  
\prod_{j=1}^{\ell} 
\left(\sum_{(\T, \sigma) \in \mathcal{C}(\abs{R_j})} \mathfrak{Int}(f, (\T,\sigma))\right),
\end{split}
\end{equation}
where
$$
\mathfrak{Int}(f, (\T,\sigma)):=\int_{\mathcal X^{\abs{\T } }} \left(
f(\T)  \text{sgn}(\sigma) \prod_{q\in \textrm{Range}(\T)} K(x_q,x_{\sigma(q)})\right)d\x.
$$

Combining   \eqref{moment_cummu} \eqref{mom-permu} and \eqref{findcumu}, 
we obtain the following formula for the cumulants of multivariate linear statistics of general \de processes.
\begin{lem}\label{soss}
\begin{equation}\label{express cumu}
\begin{split}
	&Q_m\left(\sum_{(x_1,\ldots,x_k) \in \Phi_{*}^k} f(x_1,\ldots, x_k)\right)\\
	=&\sum_{(\T, \sigma) \in \mathcal C{(m)}} \int_{\mathcal X^{\abs{\T}}} f(\T)  {\operatorname{sgn}(\sigma)} \prod_{q \in\operatorname{Range}(\T) } K(x_q,x_{\sigma(q)})d\x.
\end{split}
\end{equation}\end{lem}
For $k=1$, \eqref{express cumu} gives the following well-known formula  (Formula (2.7) in \cite{AS2}), 
\begin{equation}\label{cummufor1}
\begin{split}
&Q_m\left(\sum_{x\in \Phi} f(x)\right)\\=&\sum_{\ell=1}^m \sum_{(n_1, \ldots, n_{\ell}): \sum_{j=1}^{\ell} n_j=m, n_j\geq 1, \forall j} \frac{(-1)^{\ell-1}}{\ell}\frac{m!}{n_1!\ldots n_{\ell}!}\\
& \int_{\mathcal X^{\ell}} f^{n_1}(x_1)\cdots 
f^{n_{\ell}}(x_{\ell}) K(x_1,x_2)\cdots K(x_{\ell-1},x_{\ell})  K(x_{\ell},x_1)d\x.
\end{split}
\end{equation}
Indeed, by the definition of $S(m,1)$,
each $\T\in S(m,1)$ corresponds to one way of assigning $m$ different balls into $\ell$ indistinguishable urns for some $\ell$. Hence the $\T$-graph  itself has $\ell$ components and also  partitions the set $[m]$ into $\ell$ components.  Thus, to ensure the $(\T,\sigma)$-graph is connected, 
different components have to be linked through  $\sigma\in \text{Sym}(\text{Range}(\T))$, which implies that $\sigma $ has to be a cyclic permutation of length $\ell$. 
As an example, suppose $k=1$, $m=5$ and $\T=\{1,2,3,3,3\}$, then $\abs{\T}=3$ and $\sigma$ has to be  (123) or (132) to obtain a connected $(\T, \sigma)$-graph. 

For later reference, we introduce a few more concepts. 
\begin{defn}\label{def:con}
For a $(\T, \sigma)$-graph, we 
say $T_{i,j}$ is a \emph{connection point} if at least one of the two conditions are satisfied:
\begin{itemize}
\item  $\sigma(T_{i,j})\notin T_i$.
\item There exists an $i'\neq i$ such that $T_{i,j}\in T_{i'}$.
\end{itemize}	Equivalently, using the graphical representation of a $(\T,\sigma)$-graph, $T_{i,j}$ is a connection point if $(i,j)$ is connected to some vertex in $\{(i',j'):i'\neq i,1\leq j'\leq k\}$ by a red edge, either solid or dotted.
\end{defn}
Note that, if the $(\T,\sigma)$-graph is connected, then for each $i$, there must exist at least one connection point $T_{i,j}$. 
\begin{defn}\label{def:red}
We say a $(\T,\sigma)$ pair is \emph{reducible} if its $(\T,\sigma)$-graph is connected and there exists an $i\in [m]$ and a $j\in [k]$ such that 
\begin{itemize}
\item $T_{i,j}$ is the only connection point in $T_i$.
\item $\sigma(x)=x, \forall x\in T_i-\{T_{i,j}\}$.
\end{itemize}
If the above two conditions hold,  then we say  the $(\T,\sigma)$-graph \emph{breaks} at $T_{i,j}$ and $T_{i,j}$ is a \emph{break point}. 
Equivalently, $(\T,\sigma)$-graph is reducible if it is connected and there exists some $(i,j)$ which is the only vertex in $\{(i,j):1\leq j\leq k\}$ that can have red edge(s) connecting with other vertices. We say a $(\T,\sigma)$ pair is \emph{irreducible} if it is not reducible. 
\end{defn}
We define $\mathfrak{I}(m)$ to be the set of all $(\T,\sigma)\in \mathcal{C}(m)$ that are irreducible, i.e., 
\begin{equation}\label{defirm}
\mathfrak{I}(m):=\{(\T,\sigma)\in \mathcal{C}(m): (\T,\sigma) \mbox{ is irreducible}\}.
\end{equation}
An example of the reducible graph is given by the right panel of Figure \ref{Example2},  while the left and the middle ones in Figure \ref{Example2} are irreducible. 
\begin{defn}\label{defcircle}
We say a $(\T,\sigma)\in \mathfrak{I}(m)$ is  \emph{circle-like}  if 
for each $1\leq i\leq m$, there are exactly two distinct numbers $1\leq i_1\neq i_2\leq k$ such that each of $(i,i_1)$ and $(i,i_2)$ has exactly one red edge and the red edge is connected to a vertex in $\{(i',j'):i'\neq i, 1\leq j' \leq k\}$, and all other vertices, i.e., those not in the set $\{(i,i_1), (i,i_2):1\leq i\leq m\}$, have no red edge. 
\end{defn} 
The following proposition explains the name `circle-like'. 
\begin{prop}\label{prop:circle-like}
Let $(\T,\sigma)$ be  circle-like.  Then 
there exists a cyclic permutation $p$ of $\{1,\ldots, m\}$  such that, 
for each $1\leq i\leq m$, there exist two distinct indices $i_1$ and $i_2$ and that $(i,i_2)$ is connected with $(p(i),p(i)_1)$ with a red edge. 
\end{prop}
\begin{proof}
Note that, by the definition of being circle-like, if we contract 
all vertices in $\{(i,j): 1\leq j\leq k\}$ into a single vertex (and give it label $i$), then we will obtain a connected graph with $m$ vertices such that each vertex has degree 2, which is then necessarily a circle of size $m$. Fix a direction of the circle, suppose the label of these vertices are $a_1, \ldots, a_m$. Then we can define a permutation $p$ such that $p(a_i)=a_{i+1}$ where $a_{m+1}:=a_1$.  In addition,  by reordering $i_1$ and $i_2$ for each $1\leq i\leq m$ if needed, we can assume that $(a_i, (a_i)_2)$ is connected with $(a_{i+1}, (a_{i+1})_1)$ for all $1\leq i\leq m$ with a red edge.  This 
completes the proof.
\end{proof}

As a remark, we will see that in the proof of Theorem \ref{clt2} for the degenerate case, the collections of the cycle-like $(\T,\sigma)$-graph will provide the leading order term for the cumulants of multivariate linear statistics, which will eventually yield the 2nd order Wiener chaos.

\section{Properties of the spectral projection kernel}
In this section, we first review some basic facts for the spectral projection kernel.  Then we will derive several integral lemmas which provide the key estimates   to prove the main results. 
\subsection{Preliminaries}
It's well-known that  the kernel for the spectral orthogonal projection $K_n: L^2(S^d) \to \mathcal H_n(S^d)$ satisfies  (Theorem 2.9 in \cite{AH})
\begin{equation}\label{knpn}
K_{n}(x,y)=  \frac{k_{n}}{s_d}P_{n}(\cos \d(x,y))=
\frac{k_{n}}{s_d}P_{n}(x\cdot y),
\end{equation}
where 
$\d(x,y) \in [0,\pi]$ is the geodesic distance which is the angle between the vectors $x, y\in S^d$,  $P_{n}$ is  the Legendre polynomial of degree $n$ in $d$ dimension,  $k_n$ is the dimension of $\mathcal H_n$ given in \eqref{knquan}
and $s_d=2\pi^{\frac{d+1}2}/\Gamma(\frac{d+1}2)$ is the surface area of $S^d$. Since both $x$ and $y$ are on the unit sphere, then we can rewrite $\cos \d(x,y)=x\cdot y$ as the inner product between $x$ and $y$.

We also write 
$$P_{n}(x,y):=P_{n}(\cos \d(x,y))=P_{n}(x\cdot y).$$ 
By the fact that $P_n(1)=1$ \cite{AH},   one has the identity
\beq \label{ondiagonalkernel}K_{n}(x,x)=  \frac{k_{n}}{s_d}.\eeq
The kernel $K_n(x,y)$ satisfies the reproducing property, 
\begin{equation}\label{reproduce1}
\int_{S^d} K_n(x_1,x_2)K_n(x_2,x_3)dx_2=K_n(x_1,x_3). 
\end{equation}
When $x_1=x_3$, \eqref{reproduce1} reads, 
\beq\label{sdsdsdsd}
\int_{S^d} K_n^2(x_1,x_2) dx_2=\frac{k_{n}}{s_d}, 
\eeq
and thus  we have
\begin{equation}\label{reproduce2}
\int_{(S^d)^2} K_n^2(x_1,x_2) dx_1dx_2=k_{n}.
\end{equation}
For $P_{n}$, two basic  properties are  \cite{AH}, 
\begin{equation}\label{pnpro1}
P_{n}(x)=(-1)^n P_{n}(-x)
\end{equation}
and 
\begin{equation}\label{pnprop2}
\abs{P_{n}(x)}\leq 1,\,\, \forall x \in [-1,1]. 
\end{equation}
By \eqref{ondiagonalkernel} and the reproducing property \eqref{reproduce1},  we obtain that
\begin{equation}\label{reproduce3}
\int_{S^d} P_{n}(x_1,x_2)P_{n}(x_2,x_3)dx_2=\left(\frac{k_{n}}{s_d}\right)^{-1}P_n(x_1,x_3),
\end{equation}
and by \eqref{sdsdsdsd}, we have 
\begin{equation}\label{reproduce4}
\int_{S^d} P_{n}^2(x_1,x_2)dx_2=\left(\frac{k_{n}}{s_d}\right)^{-1}. 
\end{equation}
For $0\leq \theta\leq \pi/2$, one has 
the Hilb's asymptotics  for the Legendre polynomials (by taking $\alpha=\beta=\frac{d-2}2$ in  \cite[Theorem 8.21.12]{So}),
\begin{equation}\label{hilbs}
\begin{split}
P_{n}(\cos \theta)=& \Gamma\left(\frac{d}{2}\right) \left(\frac{\theta}{\sin\theta}\right)^{\frac{1}{2}} \left(\frac{1}{2}(n+\frac{d-1}{2} ) \sin \theta \right)^{-\frac{d-2}{2}}
J_{\frac{d-2}{2}}\left((n+\frac{d-1}{2})\theta\right)\\&+R_{n}(\theta),
\end{split}
\end{equation}
where 
$J_{\frac{d-2}{2}}$ is the Bessel function of  order $\frac{d-2}{2}$.  	 And the error term satisfies the estimates:
\[R_n(\theta)=\begin{cases}
{\theta}^2 O(1)&\text{$0\leq \theta \leq {c}{n^{-1}} $ },\\
{\theta}^{\frac{3-d}{2}}O(n^{-\frac{1+d}{2}}) &
\text{${c}{n^{-1}}\leq \theta \leq {\pi}/{2} $ }, 
\end{cases}\] where $c$ is some constant independent of $n$. 

For the Bessel function, 
$J_{\frac{d-2}{2}}$  is bounded on the positive real line and has the  expansion (Formula (1.71.1) in \cite{So}),
\beq
J_{\frac{d-2}{2}}(x)=\sum_{j=0}^{\infty} \frac{(-1)^j}{j!\Gamma(j+\frac{d}{2})} \left(\frac{x}{2}\right)^{2j+\frac{d-2}{2}}.
\eeq
Furthermore, it admits the  asymptotic expansion (Formula (1.71.7) in \cite{So}),
\beq \label{j0}
J_{\frac{d-2}2}(x)= \sqrt{\frac{2}{\pi x}}\cos\left(x-(d-1)\frac{\pi}{4}\right) +o(x^{-1}) \quad \mbox{as}\,\, \, x\to +\infty.
\eeq
Now we define a function $p_n(\theta)$ for $\theta \in [0,\pi]$ as follows. 
For $0\leq \theta\leq \pi/2$, we define
\begin{equation}\label{pndef}
\begin{split}
p_{n}(\theta):=&\Gamma\Big(\frac{d}{2}\Big)\left(\frac{\theta}{\sin \theta}\right)^{1/2}\left(\frac{1}{2}(n+\frac{d-1}{2}) \sin \theta \right)^{-(d-2)/2} \\
& \times \sqrt{\frac{2}{\pi (n+(d-1)/2)\theta }} \cos\left((n+(d-1)/2)\theta-(d-1)\frac{\pi}{4} \right)\\
=& \Gamma\Big(\frac{d}{2}\Big)\Big(\frac{2^{d-1}}{\pi}\Big)^{1/2} 
\Big(n+(d-1)/2\Big)^{-(d-1)/2}\big(\sin \theta\big)^{-(d-1)/2}\\
& \times 
\cos\left((n+(d-1)/2)\theta-(d-1)\frac{\pi}{4} \right);
\end{split}
\end{equation}
for $\pi/2<\theta \leq \pi$, we define $$
p_n(\theta):=(-1)^n p_n(\pi-\theta).
$$
Combining \eqref{hilbs} and \eqref{j0}, for $0\leq \theta\leq \pi$, we have the estimates,
\begin{equation}\label{pnes}
\abs{P_{n}(\cos \theta)-p_{n}(\theta)} \leq C\left( \min\{n\theta, n(\pi-\theta)\}^{-d/2} \wedge 1\right), 
\end{equation}
\beq \label{control pn}
\abs{P_{n}(\cos \theta)}\leq C\left( \min\{n\theta, n(\pi-\theta)\}^{-(d-1)/2} \wedge 1\right), 
\eeq
and 
\beq \label{control pn2}
|p_n(\theta)|\leq C\left( \min\{n\theta, n(\pi-\theta)\}^{-(d-1)/2} \wedge 1\right).
\eeq

\subsection{Integral estimates}
Now we  will prove  several  lemmas involving the integrals of the kernel $K_n$.  They will be one of  the main technical ingredients in the proofs of our main results. 

We will use the spherical coordinate system $ (\theta, \phi_1, \ldots, \phi_{d-1})$ for $S^d$, where $\theta, \phi_1, \ldots, \phi_{d-2}$ range over [0, $\pi$] and $\phi_{d-1}$ ranges over [0,$2\pi$]. Here, $\theta$ is the  arc length from the point $(\theta, \phi)$ to the origin of the coordinate system. For simplicity,  we will use $\phi$  as a shorthand for $(\phi^1, \ldots, \phi^{d-1})$, and thus the range of $\phi$ is $\Omega:=[0,\pi]^{d-2}\times [0,2\pi]$.  Then the volume element for $S^d$ with respect to the standard round metric  is  $$dx=\hat{J}(\theta, \phi)d\theta d\phi,$$
where 
$$\hat{J}(\theta, \phi)=\sin^{d-1}(\theta)\sin^{d-2}(\phi_1)\cdots\sin(\phi_{d-2}).$$ We define $$J(\phi):= \sin^{d-2}(\phi_1)\cdots\sin(\phi_{d-2}),$$ and thus we can rewrite $$dx=\sin^{d-1}(\theta) J(\phi)d\theta d\phi. $$
The first lemma concerns the integration of a function against $K_n^2$.
\begin{lem}\label{sc}
For any  bounded function $f(x,y)$, we have
\beq\label{sum_fxy}
\begin{split}
&	\lim\limits_{n\to \infty} \frac{1}{k_{n}}
\int_{S^d}\int_{S^d} f(x,y)K_n^2(x,y)dxdy
\\
=&\frac{2^{d-1}}{\Gamma(d)\pi} \Big(\frac{\Gamma\big(\frac{d}{2}\big)}{s_d}\Big)^2 \int_{S^d} \int_0^{\pi} 
\int_{\Omega} f(x,x+(\theta, \phi))   J(\phi) 	d\phi d\theta dx.\\
=&\frac{2^{d-1}}{\Gamma(d)\pi} \Big(\frac{\Gamma\big(\frac{d}{2}\big)}{s_d}\Big)^2 \int_{S^d} \int_{S^d}  \frac{f(x,y)}{\sin^{d-1} (\arccos (x\cdot y)) } dxdy.
\end{split}
\eeq
\end{lem}

The next two lemmas give upper bounds  on the integration of the product of several $K_n's$. 
\begin{lem}\label{productcontrol}
For any $r\in \N, r\geq 2$, 
\begin{equation}
\int_{(S^d)^r} \Big|  \prod_{i=1}^r K_n(x_i,x_{i+1}) \Big|dx_1\cdots dx_r=
O(n^{\frac{(d-1)r}{2}}),
\end{equation}
where $x_{r+1}$ is set to be $x_1$.
Equivalently, 
\begin{equation}
\int_{(S^d)^r} \Big|  \prod_{i=1}^r P_n(x_i,x_{i+1}) \Big|dx_1\cdots dx_r=
O(n^{-\frac{(d-1)r}{2}}).
\end{equation}
\end{lem}

\begin{lem}\label{keylemma}
For any $r\in \N, r\geq 3$ and bounded measurable function $h$ of $r$ variables, 
\begin{equation}
\int_{(S^d)^r} h(x_1,\ldots, x_r)\prod_{i=1}^r P_n(x_i,x_{i+1})dx_1\cdots dx_r=o(n^{-\frac{(d-1)r}{2}}),
\end{equation}
where $x_{r+1}$ is set to be $x_1$.  
\end{lem}

We now give the  proofs of Lemmas \ref{sc}-\ref{keylemma}. 
\begin{proof}[Proof of Lemma \ref{sc}] 
By the boundedness of $f$, without loss of generality, we may assume that $f$ is nonnegative. For $x,y\in S^d$,  we build a spherical coordinate system $(\theta,\phi)$ with $x$ being the north pole and write $y$ as $x+(\theta,\phi)$. 
By the facts  that $K_n(x,y)=k_{n}P_n(\cos \theta)/s_d$ and
$\abs{P_{n}\cos(\theta)}=\abs{P_{n}\cos (\pi-\theta)}$, we have 
\begin{equation}\label{Varlfn}   
\begin{split}
	&\int_{S^d}\int_{S^d} f(x,y)K_n^2(x,y)dxdy \\
	=&\Big( \frac{k_{n}}{s_d} \Big)^2\int_{S^d}\int_{0}^{\pi}\int_{\Omega} f(x,x+(\theta,\phi)) P_n(\cos \theta)^2  \hat{J}(\theta,\phi)   d\phi d\theta dx\\
	=&\Big (\frac{k_{n}}{s_d}\Big)^2 \left(\int_{S^d}\int_0^{\frac{\pi}{2}}\int_{\Omega} f(x,x+(\theta,\phi)) P_n(\cos \theta)^2 \hat{J}(\theta,\phi)   d\phi d\theta dx \right.\\
	&\quad\quad\quad\left.+  \int_{S^d}  \int_0^{\frac{\pi}{2}} \int_{\Omega} f(x,x+(\pi-\theta,\phi))  P_n(\cos \theta)^2 \hat{J}(\theta,\phi)   d\phi d\theta dx \right)\\
	:=& \Big (\frac{k_{n}}{s_d}\Big)^2 (I_1+I_2).
\end{split}
\end{equation}

We will analyze   $I_1$ and $I_2$ by a series of approximations.   We only give details for $I_1$, and $I_2$ follows from the same arguments. By Hilb's asymptotic \eqref{hilbs}, one has 
\beq \label{cov}
\begin{split}
P_n(\cos \theta)^2=&\Gamma\Big(\frac{d}{2}\Big)^2 \left(  \frac{1}{2}(n+\frac{d-1}{2}) \sin \theta \right)^{-(d-2)}\Big(\frac{\theta}{\sin \theta}\Big)J_{\frac{d-2}{2}}\left((n+\frac{d-1}{2})\theta \right)^2\\&+\hat{R}_n(\theta), \end{split}
\eeq
where
\[\hat{R}_n(\theta)=\begin{cases}
{\theta}^2 O(1)&\text{$0\leq \theta \leq {c}/{n} $ },\\
{\theta}^{2-d} O(n^{-d}) &
\text{${c}/{n}\leq \theta \leq {\pi}/{2} $ }.
\end{cases}\]
We now define
\begin{equation}\label{i1-i3}
I_3=	    	\int_{S^d}   \int_{0}^{\frac{\pi}{2}} \int_{\Omega}   \theta J_{\frac{d-2}{2}}\Big((n+\frac{d-1}{2})\theta \Big)^2 f(x,x+(\theta,\phi))   J(\phi) d\phi d\theta dx.
\end{equation}
By \eqref{Varlfn}  and \eqref{cov} there exists $C>0$ such that 
\beq \label{remainder}
\abs{I_1- \Gamma\Big(\frac{d}{2}\Big)^2 \left(  \frac{1}{2}(n+\frac{d-1}{2})  \right)^{-(d-2)} I_3}\leq Cn^{-d}.
\eeq
By   \eqref{j0}, for any $\ep\in(0, 1)$, there exists an $M>0$ large enough such that for $x>M$, we have
\begin{equation} \label{compare j0}
\begin{split}
	&(1-\ep) \frac{2}{\pi x} \cos^2 \Big(x-(d-1)\frac{\pi}{4}\Big) -x^{-\frac{3}{2}} \\
	\leq  &J_{\frac{d-2}{2}}(x)^2 \\
	\leq  &(1+\ep) \frac{2}{\pi x} \cos^2 \Big(x-(d-1)\frac{\pi}{4}\Big)+x^{-\frac{3}{2}}.
\end{split}
\end{equation}
Now we split $I_3$ into two terms, 
\begin{equation}\label{defi3}
\begin{split}
	I_3=&\int_{S^d} \int_{0}^{{M}/{n}} \int_{\Omega}  J_{\frac{d-2}{2}}\left((n+\frac{d-1}{2})\theta \right)^2 
	f(x,x+(\theta,\phi))  \theta J(\phi
	)d\phi d\theta dx \\
	&+\int_{S^d}   \int_{{M}/{n}}^{\frac{\pi}{2}} \int_{\Omega} J_{\frac{d-2}{2}}\left((n+\frac{d-1}{2})\theta \right)^2 f(x,x+(\theta,\phi))  \theta J(\phi)d\phi d\theta dx \\
	:=&I_4+I_5.
\end{split}
\end{equation}
For $I_4$, by the boundedness of $f$ and $J_{\frac{d-2}{2}}$,  there exists some $C>0$ such that 
\beq \label{i4}
\abs{I_4}\leq  CM^2/{n^2}. 
\eeq
For $I_5$, it holds trivially that $(n+\frac{d-1}{2})\theta>M$ for $\theta >{M}/{n}$. Hence, we can apply the estimates \eqref{compare j0}  for  $J_{\frac{d-2}{2}}((n+\frac{d-1}{2})\theta)$. We set
\begin{equation}\label{defi6}
\begin{split}
	I_6=&\int_{S^d}  \int_{{M}/{n}}^{\frac{\pi}{2}} \int_{\Omega} \frac{2}{\pi (n+(d-1)/2)\theta}
	f(x,x+(\theta,\phi))\\
	&	        	\times \cos^2\Big((n+\frac{d-1}{2})\theta -(d-1)\frac{\pi}{4} \Big) \theta  J(\phi)d\phi d\theta dx.
\end{split}
\end{equation}
Combining \eqref{compare j0}, \eqref{defi6} and the following estimate 
$$
\int_{S^d}   \int_{{M}/{n}}^{\frac{\pi}{2}} \int_{\Omega} \left((n+\frac{d-1}{2}\theta)\right)^{-3/2}
f(x,x+(\theta,\phi))  \theta J(\phi) d\phi d\theta dx
\leq  C n^{-3/2},
$$ we have that
\beq \label{i5 and i6}
(1-\ep)I_6 -Cn^{-\frac{3}{2}}   \leq I_5 \leq (1+\ep) I_6+Cn^{-\frac{3}{2}}.
\eeq
By Riemann-Lebesgue lemma, for any fixed $x$ and $\phi$, one has 
\beq
\begin{split}
&\lim\limits_{n\to\infty}	\int_{M/n}^{\frac{\pi}{2}}  f(x,x+(\theta,\phi)) \cos^2\Big((n+\frac{d-1}{2})\theta -(d-1)\frac{\pi}{4} \Big) d\theta \\
=& \lim\limits_{n\to\infty} \int_{M/n}^{\frac{\pi}{2}}  f(x,x+(\theta,\phi))  \Big(\frac{1}{2}+\frac{\cos((2n+d-1)\theta  -(d-1)\frac{\pi}{2}) }{2}\Big) d\theta \\
=&\frac{1}{2} \int_{0}^{\frac{\pi}{2}}  f(x,x+(\theta,\phi))  d\theta. 
\end{split}
\eeq
Therefore,  the bounded convergence theorem implies that
\begin{equation}
\begin{split}
	&	\lim\limits_{n\to\infty}	\int_{S^d}\int_{\Omega} \int_{M/n}^{\frac{\pi}{2}}  f(x,x+(\theta,\phi)) \cos^2\Big((n+\frac{d-1}{2})\theta -\frac{\pi}{4}\Big ) J(\phi) d\theta d\phi dx\\
	=&  \frac{1}{2}
	\int_{S^d}\int_{\Omega} \int_{0}^{\frac{\pi}{2}}  f(x,x+(\theta,\phi))   J(\phi)d\theta d\phi dx. 
\end{split}
\end{equation} This implies that 
\beq\label{claim45}\lim\limits_{n\to \infty} n{I_6}=\frac 1\pi\int_{S^d}\int_{\Omega} \int_{0}^{\frac{\pi}{2}}  f(x,x+(\theta,\phi))   J(\phi)d\theta d\phi dx:=I_7.\end{equation} 
Now, combining \eqref{defi3}, \eqref{i4} and \eqref{i5 and i6} and \eqref{claim45}, we have 	\begin{equation}\label{i3/i7}
(1-\ep)I_7\leq \liminf_{n\to \infty} {nI_3}  \leq \limsup_{n\to \infty}  {nI_3}{}  \leq (1+\ep)I_7.
\end{equation}
Since \eqref{i3/i7} holds for all $\ep\in(0,1)$ while $I_3$ and $I_7$ don't depend on $\ep$, by sending $\ep\to 0$,   we have
\begin{equation}\label{limi3i7}
\lim\limits_{n\to\infty} n{I_3}={I_7}.
\end{equation}
Combining 
\eqref{remainder} and \eqref{limi3i7},  we have
$$
\lim_{n\to \infty} {n^{d-1}I_1}=\Gamma\Big(\frac{d}{2}\Big)^2 2^{d-2}{I_7}.
$$
By the same argument with $\theta$ replaced by $\pi-\theta$, we get a similar limit
$$
\lim_{n\to \infty}{n^{d-1}I_2}=\Gamma\Big(\frac{d}{2}\Big)^2 2^{d-2}{I_8},
$$
where $I_8$ is defined similarly to $I_7$ as 
$$I_8=  \frac{1}{\pi } \int_{S^d}  \int_{\Omega}  \int_{\frac{\pi}{2}}^{\pi} f(x,x+(\theta,\phi)) J(\phi)d\theta d\phi  dx. $$
By the fact $k_{n}\sim 2n^{d-1}/\Gamma(d)$,  we get  	\begin{equation}
\begin{split}
&\lim_{n\to\infty}\frac{1}{k_{n}}
\int_{S^d}\int_{S^d} f(x,y)K_n^2(x,y)dxdy\\
=&\lim_{n\to\infty}\left(\frac{k_{n}}{s_d^2}\right)(I_1+I_2)\\
= & \lim_{n\to\infty}\frac{2n^{d-1}}{\Gamma(d)} \frac{1}{s_d^2}\Gamma\Big(\frac{d}{2}\Big)^2 2^{d-2}n^{-(d-1)}(I_7+I_8)\\
= & \frac{\Gamma\Big(\frac{d}{2}\Big)^2 2^{d-1}}{\pi \Gamma(d) s_d^2}  \int_{S^d}  \int_0^{\pi} \int_{\Omega}  f(x,x+(\theta,\phi)) J(\phi)d\phi d\theta  dx .
\end{split}
\end{equation}
This completes the proof of Lemma \ref{sc}. 

\end{proof}
As a remark, the proof of Lemma \ref{sc} actually shows that for almost all $x$, we have
\beq\label{rmk:lem2}
\begin{split}
&	\lim\limits_{n\to \infty} \frac{1}{k_{n}}
\int_{S^d} f(x,y)K_n^2(x,y)dy
\\=&\frac{2^{d-1}}{\Gamma(d)\pi} \Big(\frac{\Gamma\big(\frac{d}{2}\big)}{s_d}\Big)^2 \int_{S^d}  \frac{f(x,y)}{\sin^{d-1} (\arccos (x\cdot y)) } dy.
\end{split}
\eeq

\begin{proof}[Proof of Lemma \ref{productcontrol}]
To prove Lemma \ref{productcontrol}, we recall \eqref{control pn} where we have   
\begin{equation}\label{1pn}
P_n(x,y)\leq Cn^{-\frac{d-1}2}(\min\{ \d(x,y), \pi-\d(x,y)\})^{-\frac{d-1}{2}}.
\end{equation}
For $x_1,\ldots ,x_d\in S^d$,  let $\alpha_{i,j}=\d(x_i,x_j)$ be the geodesic distance which is the angle between $x_i$ and $x_j$ and let $\beta_{i,j}=\min\{\alpha_{i,j},\pi-\alpha_{i,j}\}$. We now claim 
\begin{equation}\label{betaineq}
\beta_{1,3}\leq \beta_{1,2}+\beta_{2,3}.
\end{equation}
To prove \eqref{betaineq}, we consider four possible cases. 
\begin{itemize}
\item If $\alpha_{1,2} <\pi/2$ and $\alpha_{2,3}\leq \pi/2$, then we have
$$
\beta_{1,2}+\beta_{2,3}=\alpha_{1,2} +\alpha_{2,3}\geq \alpha_{1,3} \geq \beta_{1,3}.
$$
Here, the first inequality follows from triangle inequality.
\item If $\alpha_{1,2} <\pi/2$ and $\alpha_{2,3}\geq \pi/2$, then by symmetry of the sphere, if we set $x_3':=-x_3$ (the reflection of $x_3$ through the origin of $\mathbb R^{d+1}$), we have
$$
\beta_{1,2}+\beta_{2,3}=\alpha_{1,2} +\pi-\alpha_{2,3}=\d(x_1,x_2)+\d(x_2,x_3')\geq 
\d(x_1,x_3')\geq \beta_{1,3}.
$$
\item The case  $\alpha_{1,2} \geq \pi/2$ and $\alpha_{2,3}< \pi/2$ can be analyzed 
similarly to the  second case. 
\item If $\alpha_{1,2} \geq \pi/2$ and $\alpha_{2,3}\geq \pi/2$, then by setting 
$x_2':=-x_2$, we have
$$
\beta_{1,2}+\beta_{2,3}=\d(x_1,x_2')+\d(x'_2,x_3)\geq 
\d(x_1,x_3)\geq \beta_{1,3}.
$$
\end{itemize}
The inequality \eqref{betaineq} implies that
\begin{equation*}
\beta_{1,2}\beta_{2,3}=\max\{\beta_{1,2},\beta_{2,3}\}\min\{\beta_{1,2},\beta_{2,3}\}
\geq \frac{\beta_{1,3}}{2}\min\{\beta_{1,2},\beta_{2,3}\},
\end{equation*}
which gives
\begin{equation}\label{betafinal}
\begin{split}
(\beta_{1,2}\beta_{2,3})^{-(d-1)/2}&\leq C \beta_{1,3}^{-(d-1)/2}
\min\{\beta_{1,2},\beta_{2,3}\}^{-(d-1)/2}\\
&\leq 
C\beta_{1,3}^{-(d-1)/2}\left(\beta_{1,2}^{-(d-1)/2} +\beta_{2,3}^{-(d-1)/2}\right).
\end{split}
\end{equation}
By \eqref{1pn} and \eqref{betafinal}, for any fixed $x_1$ and $x_3$,  we have
\begin{equation}\label{cancelpn}
\begin{split}
&    \int_{S^d}\Big|P_{n}(x_1,x_2)P_n(x_2,x_3)\Big|dx_2 \\
\leq &
Cn^{-(d-1)}      \int_{S^d} (\beta_{1,2}\beta_{2,3})^{-(d-1)/2}dx_2\\
\leq &Cn^{-(d-1)} \beta_{1,3}^{-(d-1)/2}      \int_{S^d}
\left(\beta_{1,2}^{-(d-1)/2} +\beta_{2,3}^{-(d-1)/2}\right)dx_2\\
\leq &Cn^{-(d-1)}\beta_{1,3}^{-(d-1)/2}\left(\int_0^{\pi} \beta_{1,2}^{-(d-1)/2}\sin^{d-1} (\alpha_{1,2} )d\alpha_{1,2} \right.\\
&\left.+
\int_0^{\pi} \beta_{2,3}^{-(d-1)/2}\sin^{d-1} (\alpha_{2,3})d\alpha_{2,3}\right)\\
\leq &C n^{-(d-1)}\beta_{1,3}^{-(d-1)/2}. 
\end{split}
\end{equation}
Using \eqref{cancelpn} $r-2$ times to integrate out the variables $x_2,\ldots, x_{r-1}$, we get 
\begin{equation}
\begin{split}
&\int_{(S^d)^r} \Big| \prod_{i=1}^r P_n(x_i,x_{i+1}) \Big | dx_1\cdots dx_r\\\leq& Cn^{-(d-1)r/2}\int_{S^d}\Big( \int_0^{\pi} \beta_{1,r}^{-(d-1)}\sin^{(d-1)}(\alpha_{1,r} )d\alpha_{1,r} \Big)dx_1\\
\leq &C n^{-(d-1)r/2}.
\end{split}
\end{equation}
This proves 
Lemma \ref{productcontrol}. 

\end{proof}

\begin{proof}[Proof of Lemma \ref{keylemma}]
As in the proof of Lemma \ref{productcontrol}, let  $\alpha_{i,i+1}$ be the angle between $x_i$ and $x_{i+1}$  and set  $\beta_{i,i+1}=\min\{\alpha_{i,i+1},\pi-\alpha_{i,i+1}\}$. 
Recall the function $p_n$ defined in \eqref{pndef}, by \eqref{pnes},
\begin{equation}\label{diffe}
\abs{P_n(x_i,x_{i+1})-p_n(\alpha_{i,i+1})} =\abs{P_n(\cos \alpha_{i,i+1})-p_n(\alpha_{i,i+1})}\leq C (n\beta_{i,i+1})^{-d/2}.
\end{equation}
We can write
\begin{equation}\label{hPn}
h(x_1,\ldots, x_r)\Pi_{i=1}^r P_n(x_i,x_{i+1})	=h(x_1,\ldots, x_r)    \Pi_{i=1}^r p_n(\alpha_{i,i+1})+I_r
\end{equation}
where  the error term $I_r$ is bounded from above as
\begin{equation}\label{bdim2}
\begin{split}
I_r \leq &	C\abs{		h(x_1,\ldots, x_r)}\sum_{j=1}^r (n\beta_{j,j+1})^{-d/2}\Big(\prod_{i=1,i\neq j}^r \left(\abs{P_n(\cos \alpha_{j, j+1})}+\abs{p_n(\alpha_{j,j+1})}\right)\Big)\\
\leq &				Cn^{-\frac{(d-1)(r-1)}{2}}n^{-d/2}  \sum_{j=1}^r \left( \beta_{j,j+1}^{-d/2} \left(\prod_{i=1,i\neq j}^r  \beta_{i,i+1}^{-(d-1)/2} \right)  \right).
\end{split}
\end{equation}The first inequality is given by the estimate \eqref{diffe} together with the  following elementary inequality: given $a_1,\ldots, a_r,  b_1, \ldots, b_r \in \mathbb R$,  one has 
$$ \left|\prod_{i=1}^r a_i-\prod_{i=1}^r b_i\right| \leqslant \sum_{j=1}^r\left|a_j-b_j\right|\left(\prod_{i=1, i\neq j}^r \left(\left|a_i\right|+\left|b_i\right|\right)\right) .
$$
The second inequality in \eqref{bdim2} is given by the estimates \eqref{control pn} and \eqref{control pn2}.

By slightly modifying the proof of Lemma \ref{productcontrol} we can  show that 
\begin{equation}\label{bdim0}
\int_{(S^d)^r}	\beta_{j,j+1}^{-d/2}\left(\prod_{i=1,i\neq j}^r  \beta_{i,i+1}^{-\frac{d-1}{2}} \right)dx_1\cdots dx_r<\infty.
\end{equation}
Combining \eqref{bdim2} and \eqref{bdim0},  we get 
\begin{equation}\label{proder}
\int_{(S^d)^r} \abs{I_r}dx_1\cdots dx_r\leq Cn^{-\frac{(d-1)r}{2}-\frac12}=o(n^{-\frac{(d-1)r}{2}}).
\end{equation}
We define a function
\begin{equation}
g(x_1,\ldots, x_r): =h(x_1,\ldots, x_r) \prod_{i=1}^r \sin^{-(d-1)/2}(\beta_{i,i+1}).
\end{equation}
The proof of Lemma \ref{productcontrol}  implies that 
the function $\Pi_{i=1}^r  \sin^{-\frac{d-1}{2}}(\beta_{i,i+1})$ is integrable over $(S^d)^{ r}$.
On the other hand, by definition of $p_n$,  we  can write 
\begin{equation}\label{hpn}
\begin{split}
& 	   h(x_1,\ldots, x_r)        \prod_{i=1}^r p_n(\alpha_{i,i+1})\\
=& (n+(d-1)/2)^{-(d-1)r/2} \times g(x_1,\ldots, x_r)\\
&\times \prod_{i=1}^r 
\left((-1)^{n\1[\alpha_{i,i+1}>\pi/2]} \cos\left((n+\frac{d-1}{2} )\beta_{i,i+1}-(d-1)\frac{\pi}{4} \right) \right). 
\end{split}
\end{equation}
When computing the integration over $x_1,\ldots, x_r$, we can build a spherical coordinate system $(\theta,\phi)$ around $x_2$ and represent  $x_1$ by $x_2+(\theta,\phi)$. Here  $\theta\in [0,\pi]$ and
$\phi$ has $d-1$ components $\phi_1,\ldots, \phi_{d-1}$.  
We claim that, for almost every (fixed)   $\phi, x_2, \ldots,  x_r$, the  integration of \eqref{hpn} over $\theta$ has the limit 
\begin{equation}\label{2cos}
\begin{split}
\lim_{n\to\infty}	&\int_{S^d} g(x_2+(\theta,\phi),x_2,\ldots, x_r)\times \\
&\prod_{i=1,r} \left( (-1)^{n\1[\alpha_{i,i+1}>\pi/2]} 
\cos\left((n+\frac{d-1}{2} )\beta_{i,i+1}-(d-1)\frac{\pi}{4} \right)   \right)d\theta =0.
\end{split}
\end{equation}
Assume \eqref{2cos} for the moment, by \eqref{hpn} and the dominated convergence theorem, we have 
\begin{equation}\label{prodj}
\int_{(S^d)^r}   h(x_1,\ldots, x_r)        \prod_{i=1}^r p_n(\alpha_{i,i+1})
dx_1\cdots dx_r= o(n^{-\frac{(d-1)r}{2}}).
\end{equation}
Lemma \ref{keylemma} now follows from \eqref{hPn}, \eqref{proder} and \eqref{prodj}. Hence it remains to prove \eqref{2cos}. 

To this end we first rewrite the product of the two $\cos(\cdots)$ factors in \eqref{2cos} as 
\begin{equation}\label{2cos2}
	\begin{split}
&\frac{1}{2} 
\cos\left((n+\frac{d-1}{2} )(\beta_{1,2}+\beta_{r,r+1})-(d-1)\frac{\pi}{2} \right)\\
+&\frac{1}{2}
\cos\left((n+\frac{d-1}{2} )(\beta_{1,2}-\beta_{r,r+1}) 
\right).
\end{split}
\end{equation}
Under the spherical coordinate system, $\alpha_{1,2}=\theta$ so that $\beta_{1,2}=\min\{\theta, \pi-\theta\}$.  Denote by $(\theta',\phi')$ the coordinate of $x_r$ in this system. By making an orthogonal transformation if necessary, we may assume that $\phi'_1=0$.

To compute $\beta_{r,r+1}$, note that 
\begin{equation}
\cos \alpha_{r,r+1}=\cos \alpha_{r,1}= x_1 \cdot x_r= \cos \theta \cos \theta'+\sin \theta \cos \phi_1 \sin \theta'.	
\end{equation} 
If neither $\theta'$ nor $\phi_1$ is  not equal to 0 or $\pi$, then $\alpha_{r,r+1}$, viewed as a function of $\theta$,  is continuously differentiable at all but finite many $\theta$'s,  and satisfies
$$
\abs{ \frac{d\alpha_{r,r+1}}{d\theta}} =
\frac{\abs{-\sin \theta \cos \theta'+ \cos \theta \cos \phi_1 \sin \theta'}}{\sqrt{1-(\cos \theta \cos \theta'+\sin \theta \cos \phi_1 \sin \theta')^2}}  < 1.
$$
Thus,  $\beta_{1,2} \pm \beta_{r,r+1}$ is piecewise differentiable in $\theta$ with a nonzero derivative.
The limit \eqref{2cos} now follows from \eqref{2cos2} and (the proof of)  the Riemann-Lebesgue lemma. 

Note that \eqref{2cos} is not true for $r=2$ where  the second $\cos(\cdots)$ factor in \eqref{2cos2} is a constant, which further implies that the integration \eqref{2cos} may tend to some constant other than 0. Thus we need the assumption $r\geq 3$.

\end{proof}

\section{Proof of Theorem \ref{mean}}
In this section we prove Theorem \ref{mean} regarding the asymptotic expansion of the mean  $\E(L_nf)$.  By \eqref{defrho1} and \eqref{rhodet}, we have
\begin{equation}\label{elnf}
\E(L_nf)=\int_{(S^d)^{ k}}f(x_1,\ldots, x_k)\det\Big(K_n(x_i,x_j)_{1\leq i,j\leq k}\Big)dx_1\cdots dx_k
\end{equation} 
We can expand the determinant as
\begin{equation*}
\begin{split}
\det\Big(K_n(x_i,x_j)_{1\leq i,j\leq k}\Big)= &\prod_{i=1}^kK_n(x_i,x_i)-
\sum_{1\leq i<j\leq k}K_n^2(x_i,x_j)\prod_{\ell \neq i,j}K_n(x_{\ell},x_{\ell})\\
& +\textrm{remainder term},
\end{split}
\end{equation*}
where the remainder term (denoted by $I_9$) is the sum of $\text{sgn}(\sigma)\Pi_{i=1}^k K_n(x_i,x_{\sigma(i)})$ over all $\sigma's \in \text{Sym}(k)$ which are neither the identity nor a transposition (a permutation which exchanges two elements and keeps all others fixed). 
Using the cycle decomposition of permutations,  \eqref{knpn} and \eqref{pnprop2}, we have the upper bound
\begin{equation}\label{i9bd}
\begin{split}
\abs{I_9}\leq &C\left(\frac{k_n}{s_d}\right)^k \left(\sum_{\sigma=(i_1j_1)(i_2j_2)  } P_n^2(x_{i_1},x_{j_1}) P_n^2(x_{i_2},x_{j_2})\right. \\
& \left. +\sum_{3\leq r\leq  k} \sum_{\sigma=(i_1\cdots i_r) } \abs{P_n(x_{i_1},x_{i_2}) \cdots P_n(x_{i_{r-1}},x_{i_r})P_n(x_{i_r},x_{i_1}) }  \right),
\end{split}
\end{equation}where  $C$ is some constant depending on $k$.

Combining \eqref{reproduce4},  the estimate $k_{n}=\Theta (n^{d-1})$,  the boundedness of $f$ and Lemma \ref{productcontrol}, we have the upper bound
\begin{equation}\label{eerctl}
\int_{(S^d)^{ k}} \abs{f(x_1,\ldots, x_k)}\abs{I_9}dx_1\cdots dx_k \leq 
Cn^{(d-1)k}(n^{-2(d-1)}+n^{-3(d-1)/2} ), 
\end{equation}which gives the error term in  \eqref{eq:mean}. 
We also have
\begin{equation}\label{emain}
\begin{split}
&\int_{(S^d)^{  k}} f(x_1,\ldots, x_k)\\  &\times \Big(\prod_{i=1}^kK_n(x_i,x_i)-
\sum_{1\leq i<j\leq k}K_n^2(x_i,x_j)\prod_{\ell \neq i,j}K_n(x_{\ell},x_{\ell}) \Big)dx_1\cdots dx_k \\
=& \left(\frac{k_n}{s_d}\right)^k \int_{(S^d)^{  k}}f(x_1,\ldots, x_k)dx_1\cdots dx_k\\
&-\left(\frac{k_n}{s_d}\right)^{k-2} \int_{(S^d)^{  2}} \sum_{1\leq i<j\leq k} f_{i,j}(x,y)K_n^2(x,y)dxdy,
\end{split}
\end{equation}
where  $f_{i,j}$ is the $(i,j)$-margin function of $f$ as defined in \eqref{ijmargin}. Applying Lemma \ref{sc} to \eqref{emain}, we will get the first two terms in \eqref{eq:mean},  which
finishes the proof of  Theorem \ref{mean}. 

\section{Proof of Theorem \ref{clt} }
\subsection{Univariate case}

The univariate linear statistics for \de processes has been understood very well. The following result proved in \cite{AS} is particularly useful.  Given a family of determinantal point processes with kernel $K_n$ and measurable bounded univariate functions $f_n$ with compact support (to ensure integrability),  let $L_nf_n$ and  $L_n\abs{f_n}$ be the linear statistics of $f_n$ and $\abs{f_n}$, respectively. Suppose that 
\beq\label{cojd}
\mathrm{\mathrm{Var}}(L_nf_n)\to \infty, \,\sup \abs{f_n}=o(\mathrm{\mathrm{Var}}(L_nf_n)^{\ep}), \,\,
\E(L_n\abs{f_n})=O(\mathrm{Var}(L_nf_n)^{\delta})
\eeq 
for any $\ep>0$ and some $\delta>0$, 
then one has the central limit theorem, $$
\frac{L_nf_n-\E(L_nf_n)}{\sqrt{\mathrm{Var}(L_nf_n)}} \xrightarrow{\textrm{d}} N(0,1).
$$
In our case, the integrability condition  holds trivially as the test function is bounded and the underlying space $S^d$ is compact. Thus, it remains to check the three conditions  in \eqref{cojd} in order to  to prove Theorem \ref{clt} for the univariate case.

Note that 		the variance of $L_nf$ is given by  
\beq \label{Var_formula}
\mathrm{\mathrm{Var}}(L_nf)=\frac{1}{2}\int_{S^d}\int_{S^d} (f(x)-f(y))^2 K_{n}^2(x,y)dxdy.
\eeq
By Lemma \ref{sc}, one immediately has the limit, 
\beq\label{limit_Variance}
\begin{split}
&	\lim\limits_{n\to \infty} \frac{\mathrm{Var}(L_n f)}{k_{n}}\\
&=\frac{2^{d-2}}{\Gamma(d)\pi} \Big(\frac{\Gamma(\frac{d}{2})}{s_d}\Big)^2 \int_{S^d} \int_0^{\pi} 
\int_{\Omega} (f(x)-f(x+(\theta, \phi)))^2 J(\phi) 	d\phi d\theta dx.\\
&=\frac{2^{d-2}}{\Gamma(d)\pi}\Big (\frac{\Gamma(\frac{d}{2})}{s_d}\Big)^2 \int_{S^d} \int_{S^d}\frac{ (f(x)-f(y))^2 }{\sin^{d-1} (\arccos (x\cdot y)) } dxdy. 
\end{split}
\eeq
By  definition \eqref{def_fi}, the 1-margin function is itself  for $k=1$, i.e.,  $F(x)=f(x)$, and thus \eqref{limit_Variance} gives the limit of variance in \eqref{q2case1} for $k=1$. 
The assumption that $F(x)$ is not constant almost everywhere 
implies the first condition $\mathrm{\mathrm{Var}}(L_nf_n)=\Theta(k_n)\to \infty$. The second condition is satisfied since $f$ is bounded. The third condition is satisfied with $\delta=1$ by the fact that 
$$\E(L_n\abs{f})=\int_{S^d}\abs{f(x)}K_n(x,x)dx=\frac{k_{n}}{s_d}\int_{S^d}\abs{f(x)}dx=\Theta(k_n).$$
This completes the proof of Theorem \ref{clt}  for the univariate case.

\subsection{Multivariate case} Now we prove Theorem \ref{clt}   for the multivariate linear statistics.  There are two steps in the proof. We will first derive the growth order of the variance $\mathrm{Var}(L_nf)=Q_2(L_nf)$, then we will prove $Q_m(L_nf)=o(Q_2(L_nf)^{\frac{m}{2}})$ for all fixed $m\geq 3$. This will imply the Gaussian limit for 
the multivariate linear statistics by the method of cumulants.

We first  introduce a notation. Given the set $A$ which is a collection of $(\T, \sigma)$-graph, we define
\beq \label{qmexpress}Q_m(L_nf,A):=
\sum_{(\T,\sigma)\in A} 
\int_{(S^d)^{\abs{\T}}} f(\T)   \text{sgn}(\sigma) \Pi_{q\in \textrm{Range}(\T)} K(x_q,x_{\sigma(q)})d\x,
\eeq where $d\x$ is the volume element involved in the integration. With such notation, we have $Q_m(L_nf)=Q_m(L_nf, \mathcal{C}(m))$ by \eqref{express cumu} (recall the definition of $\mathcal{C}(m)$ in \eqref{def*}).

We first estimate  $Q_2(L_nf)$,  which is   the variance $\mathrm{Var}(L_nf)$. We can split the expression for 
$Q_2(L_nf)$ into 3 parts:
$$
Q_2(L_nf)=Q_m(L_nf, \mathcal C(2))=Q_2(L_nf,A_1)+Q_2(L_nf, A_2)+Q_2(L_nf,A_3),
$$
where 
$A_1,A_2, A_3$ are disjoint subsets of $\mathcal{C}(2)$ defined as follows:
\begin{align*}A_1&=\{(\T,\sigma)\in  \mathcal C(2):\abs{\T}=2k, \sigma \mbox{ is a transposition, i.e. } \sigma=(ij) \mbox{ for some } i,j\},\\
A_2&=\{(\T,\sigma)\in  \mathcal C(2):\abs{\T}=2k-1, \sigma=id  \},\\
A_3&= \mathcal C(2)-A_1-A_2.
\end{align*}

\begin{lem}\label{Varctl}	
We have the following two estimates. 
\begin{enumerate}
\item 
\begin{equation}\label{q2a1a2}
\begin{split}
	&Q_2(L_nf, A_1)+Q_2(L_nf, A_2)\\
	=&\left(\frac{k_n}{s_d}\right)^{2k-2}\frac{k_n2^{d-2}}{\Gamma(d)\pi}\Big (\frac{\Gamma(\frac{d}{2})}{s_d}\Big)^2 \int_{S^d} \int_{S^d}  \frac{(F(x)-F(y))^2}{\sin^{d-1} (\arccos (x\cdot y)) } dxdy\\&+o(n^{(d-1)(2k-1)}).
\end{split}
\end{equation} 
\item  $Q_2(L_nf,A_3)=
o\left(n^{(d-1)(2k-1)}\right)$.
\end{enumerate}
\end{lem}
The limit \eqref{q2case1} now follows from Lemma \ref{Varctl}. In particular, since $F$ is not  constant almost everywhere,  we have the following estimate of the variance
\begin{equation}\label{q2ctlcase1}
Q_2(L_nf)=\Theta(n^{(d-1)(2k-1)}).
\end{equation}

\begin{proof}[Proof of Lemma \ref{Varctl} ]
We first consider $Q_2(L_nf,A_1)$. If $\abs{\T}=2k$, then $\T$ has to be
$((1,\ldots, k),(k+1,\ldots, 2k))$.
Pick any $1\leq i\leq k$ and $k+1\leq j\leq 2k$. 	Then for such $\T$ and 
$\sigma$	we have
\begin{equation*}
\begin{split}
Q_2(L_nf,(\T,\sigma))&=-\int_{(S^d)^{2k}} f(x_1,\ldots, x_k)f(x_{k+1},\ldots, x_{2k})\Big(\frac{k_n}{s_d}\Big)^{2k-2}K_n^2(x_i,x_j)d\x\\
&=-\Big(\frac{k_n}{s_d}\Big)^{2k-2} \int_{(S^d)^2} f_i(x_i)f_{j-k}(x_j)K_n^2(x_i,x_j)dx_idx_j,
\end{split}
\end{equation*}
where  the second equality is given by the definition of the $i$-margin function $f_i$ in \eqref{def_fi}.
Summing over all $i,j$, we see that $Q_2(L_nf, A_1)$ is equal to
\begin{equation}\label{Q2A1}
\begin{split}
&-\Big(\frac{k_n}{s_d}\Big)^{2k-2}\int_{(S^d)^2}  \left(\sum_{i=1}^k f_i(x)\right)  \left(\sum_{i=1}^k f_i(y)\right) K_n^2(x,y)dxdy\\
=&-\left(\frac{k_n}{s_d}\right)^{2k-2}\int_{(S^d)^2} F(x)F(y)K_n^2(x,y)dxdy.
\end{split}
\end{equation}
Now we consider  $A_2$.
Since $\T\in S(2,k)$ and $\abs{\T}=2k-1$, $\T$ has to satisfy $|T_1\cap T_2|=1$. The number of ways to choose 1 location in $T_1$ and
1 location in $T_2$ are both $k$. 
Therefore,  $Q_2(L_nf,A_2)$ equals 
\begin{equation*}
\begin{split}
&\sum_{i=1}^{k}\sum_{j=k+1}^{2k}\Big(\frac{k_n}{s_d}\Big)^{2k-1}\int_{(S^d)^{2k-1}} f(x_1, \ldots,  x_k)
f(x_{k+1}, \ldots, x_{j-1}, x_i, x_{j},\ldots, x_{2k-1})d\x\\
=&\sum_{i=1}^{k}\sum_{j=k+1}^{2k}\Big(\frac{k_n}{s_d}\Big)^{2k-1} 
\int_{S^d} f_{i}(x)f_{j-k}(x)dx\\
=&\Big(\frac{k_n}{s_d}\Big)^{2k-1}  \int_{S^d} F(x)^2dx.
\end{split}
\end{equation*}
Adding up $Q_2(L_nf,A_1)$ and  $Q_2(L_nf,A_2)$ and using \eqref{sdsdsdsd}, we have 
$$
Q_2(L_nf,A_1)+Q_2(L_nf,A_2)=\Big(\frac{k_n}{s_d}\Big)^{2k-2} \frac{1}{2}\int_{(S^d)^2} 
(F(x)-F(y))^2K_n^2(x,y)dxdy.
$$
Now \eqref{q2a1a2} follows by applying Lemma \ref{sc} to the function $(F(x)-F(y))^2$.

Now we turn to the second part of Lemma \ref{Varctl}.
We can further decompose the set $A_3$ into 3 subsets $A_4,A_5,A_6$ corresponding to $\abs{\T}=2k$ or $2k-1$ or smaller than $2k-1$. 	
For any $(\T,\sigma)\in A_4$, $\sigma$ is neither a transposition  nor identity (because $(\T,\sigma)$ has to induce a connected graph), thus there are at least three different indices $q$ such that $\sigma(q)\neq q$. By \eqref{express cumu} and Lemma \ref{productcontrol} with $r=3$,  we have 
\begin{equation}\label{bda4}
Q_2(L_nf,A_4)=O(n^{(2k)(d-1)}n^{-\frac{3(d-1)}{2}}  )=o(n^{(d-1)(2k-1)}).
\end{equation}
For any $(\T,\sigma)\in A_5$, it is not in $A_2$, i.e., $\sigma$ is not identity,  and thus there are at least two $q$'s such that $\sigma(q)\neq q$. Applying Lemma \ref{productcontrol} with $r=2$, we get 
\begin{equation}\label{bda5}
Q_2(L_nf,A_5)=O(n^{(2k-1)(d-1)}n^{-(d-1)}  )=o(n^{(d-1)(2k-1)}).
\end{equation}
For any $(\T,\sigma)\in A_6$, it's clear that if $\abs{\T}\leq 2k-2$, then for any $\sigma$, we have 
$$
\abs{Q_2(L_nf, (\T,\sigma))}\leq C\Big(\frac{k_n}{s_d}\Big)^{\abs{\T}}=O(n^{(d-1)\abs{\T}})=o(n^{(d-1)(2k-1)}).
$$ 
Hence, we have 
\begin{equation}\label{bda6}
Q_2(L_nf,A_6)=o(n^{(d-1)(2k-1)}).
\end{equation}
Combining \eqref{bda4}, \eqref{bda5} and \eqref{bda6}, we  have 
$$
Q_2(L_nf,A_3)=Q_2(L_nf,A_4)+Q_2(L_nf,A_5)+Q_2(L_nf,A_6)=o(n^{(d-1)(2k-1)}),
$$
which  completes the proof of Lemma \ref{Varctl}. 
\end{proof}
Next we will prove the estimates for the higher order cumulants. \begin{lem} \label{comc}For any $m\geq 3$, it holds that 
\beq\label{calim}Q_m(L_nf)=o(\mathrm{Var}(L_nf)^{\frac{m}{2}}),\,\,\mbox{i.e.,}\,\,
Q_m(L_nf)=o(n ^{(d-1)(km-\frac{m}{2})}).\eeq\end{lem}
This lemma will imply the convergence of the multivariate linear statistics to the Gaussian distribution \eqref{connormal}  by the method of cumulants. 
To prove Lemma \ref{comc}, we first need the following lemma.   	\begin{lem}\label{lem:tree}
Given a permutation $\sigma$, let $a(\sigma)$ be the number of  elements $q$ such that $\sigma(q)\neq q$. 
Suppose the $(\T,\sigma)$-graph is connected, then we have 
\begin{equation}\label{tree}
km-\abs{\T}+a(\sigma)\geq m-1+\1[{\sigma \neq id}]. 
\end{equation}
\end{lem}
\begin{proof} 
\eqref{tree} is essentially due to the simple fact in graph theory  that for a connected graph the number of edges is not smaller than the number of vertices minus 1. 

Before applying this fact, we note that, due to the construction of the $(\T,\sigma)$-graph, the connectivity property of the graph is not affected by removing some
redundant red edges. 
Indeed, if
a vertex has $\ell$ solid red edges, then it lies in a clique (i.e., a complete graph)  of size $\ell+1$ formed by red solid edges only. We can change this clique to a path graph by removing
$
\frac{\ell (\ell+1)}{2}-\ell =\frac{\ell (\ell-1)}{2}
$
red solid	edges  without affecting the connectivity.
After the edge removals, the number of solid red  edges becomes $km-\abs{\T}$. 

We now consider the new $(\T,\sigma)$-graph after removing some redundant red edges as described above. Note that
the total number of vertices and black edges are equal to $km$ and $(k-1)m$, respectively.  
\begin{itemize}
\item If $\sigma=id$, then there is no dotted red edge. The number of red solid edges (after the edge removals)  is equal to $km-\abs{\T}$.	
Hence, by the connectivity of the graph,  we have
$$
(k-1)m+km-\abs{\T}\geq km-1,
$$
which proves \eqref{tree}.
\item 
If $\sigma\neq id$, then we have dotted red edges.
We now perform a contraction of the graph by contracting all vertices connected by black or red solid edges into a single one. After this contraction, the number of remaining vertices is at least
$$
m-(km-\abs{\T}).
$$ These remaining vertices must be connected by dotted red edges to ensure that the $(\T,\sigma)$-graph is connected, whose number can be upper bounded by $a(\sigma)-1$. (We may remove one dotted red edge without affecting the connectivity, if the number of the vertices is $a(\sigma)$.) This implies 
$$
a(\sigma)-1\geq m-(km-\abs{\T})-1,
$$ which proves \eqref{tree}  in the case $\sigma\neq id$.
\end{itemize}
\end{proof}

Now we decompose $\mathcal{C}(m)$ into the following three subsets,  
\begin{equation}\label{}
\begin{split}
B_1&=\{(\T,\sigma)\in \mathcal{C}(m):\abs{\T}=km, a(\sigma)=m\},\\
B_2&= (\mathcal{C}(m)-B_1) \cap  \{(\T,\sigma)\in \mathcal{C}(m):\sigma=id\},\\
B_3&= (\mathcal{C}(m)-B_1) \cap  \{(\T,\sigma)\in \mathcal{C}(m):\sigma\neq id\}.
\end{split}
\end{equation}
For any $(\T,\sigma)\in B_1$,  by  the restrictions that $a(\sigma)=m\geq 3$ and $(\T,\sigma)\in \mathcal{C}(m)$ ,  the cycle decomposition of $\sigma$ must be the multiplication of one cyclic permutation of length $m$  and $(mk-m)$ cyclic permutations of length 1, e.g., $\sigma=(12\cdots m)(m+1)\cdots(km)$.  Applying Lemma \ref{keylemma} with $r=m\geq 3$, we have $$Q_m(L_nf, (\T,\sigma))=o( n^{(d-1)\abs{\T}}  n^{-\frac{(d-1)m}{2}}    )=o(n^{(d-1)(km-\frac{m}{2})}),$$ which further implies that 
\begin{equation}\label{bdb1}
Q_m(L_nf, B_1)=o(n^{(d-1)(km-\frac{m}{2})}).
\end{equation}
For any $(\T,\sigma)\in B_2$, by $m\geq 3$,  \eqref{express cumu}, \eqref{tree} and the boundedness of   $f$,  we have 
$$
Q_m(L_nf, (\T,\sigma))=O( n^{(d-1)\abs{\T}}  )=O(n^{(d-1)(km-m+1)})
=o(n^{(d-1) (km-\frac m2)}).
$$
Therefore, we get the estimate 
\begin{equation}\label{bdb2}
Q_m(L_nf, B_2)=o(n^{(d-1) (km-\frac m2)}).
\end{equation}
For any  $(\T,\sigma)\in B_3$,  by the boundedness of  $f$ and Lemma \ref{productcontrol} with $r=a(\sigma)\geq 2$,  we have 
$$
Q_m(L_nf, (\T,\sigma))  )=O(n^{(d-1)\abs{\T}}n^{-(d-1)a(\sigma)/2}  )
).
$$
If $\abs{\T}=km$, then we must have $a(\sigma)>m$ since $(\T,\sigma)\in \mathcal{C}(m)$ is connected but it is not in $B_1$. It follows that 
$$
O(n^{(d-1)\abs{\T}}n^{-(d-1)a(\sigma)/2}  )=o(n^{(d-1) (km-\frac m2)}).
$$
If $\abs{\T}<km$, then by \eqref{tree} with $\sigma\neq id$, we have 
$$
km-\abs{\T}+\frac{a(\sigma)}{2}\geq km-\abs{\T}+\frac{m-(km-\abs{\T})}{2}>\frac{m}{2}, 
$$
which implies 
$$
Q_m(L_nf,(\T,\sigma))=	O(n^{(d-1)\abs{\T}}n^{-(d-1)a(\sigma)/2}  )=o(n^{(d-1) (km-\frac m2)}).
$$
Hence, we  have 
\begin{equation}\label{bdb3}
Q_m(L_nf, B_3)=o(n^{(d-1) (km-\frac m2)}).
\end{equation}
By \eqref{bdb1}, \eqref{bdb2} and \eqref{bdb3},  for $m\geq3$ we get 
$$
Q_m(L_nf)=	Q_m(L_nf,B_1)+Q_m(L_nf,B_2)+Q_m(L_nf,B_3)=o(n^{(d-1)(km-\frac{m}{2})}).
$$
This together with \eqref{q2ctlcase1} will complete the proof of Lemma \ref{comc}, and thus the proof of Theorem \ref{clt} for $k\geq2$.

\section{Proof of Theorem \ref{clt2}}
In this section, we will prove Theorem \ref{clt2}.
We first claim that if  $f(x_1,..., x_k)$ with $k\geq 2$ satisfies \eqref{sym} and \eqref{dis}, then the $i$-margin function $f_i(x)$ is necessarily  constant for all $1\leq i\leq k$. In fact,  condition \eqref{sym} of the permutation invariance  implies that \beq\label{allequ}f_i(x)=f_1(x)\,\,\, \text{for all}\,\,\, i,\eeq which is equal to
\begin{equation*}
\begin{split}
&\int_{(S^d)^{k-1}}f(x, x_2,\ldots, x_k)dx_2\cdots dx_k\\=&
\int_{S^d}  \left( \int_{(S^d)^{k-2}} f(x,x_2,\ldots, x_k) dx_3\cdots dx_k \right)   dx_2\\
=&\int_{S^d}f_{1,2}(x,x_2)dx_2.
\end{split}
\end{equation*}
Here $f_{1,2}$ is $(1,2)$-margin function of $f$. 
Condition \eqref{dis} further
implies that the integral $\int_{S^d}f_{1,2}(x,x_2)dx_2$ is independent of $x$, i.e.,  $f_1(x)$ is a constant independent of $x$, and thus $F(x)$ is a constant.  Therefore, the limit of the variance on the right hand side of \eqref{q2case1} is degenerate.  
Without loss of generality,  we assume that the integral of $f$ is 0, i.e., $$ 
\int_{(S^d)^k} f(x_1,\ldots, x_k)dx_1\cdots dx_k=0.
$$ This is equivalent to $\int_{S^d} f_1(x)dx=0$, 
which  implies that (since $f_1$ is constant)	\begin{equation}\label{f1=0}
f_1(x)=0\,\, \text{and thus}\,\, F(x)=0  \,\,\text{for all}\,\,x\in S^d.
\end{equation}

\subsection{Calculations of the cumulants}
Again we will prove Theorem \ref{clt2} by the method of cumulants.   Recall the concepts of   break points, (ir)reducible graph and  the notation $\mathfrak{I}(m)$  (see Definitions \ref{def:con} and \ref{def:red}, and 
\eqref{defirm}), we first have 
\begin{lem}\label{anobs}
Let $f$ be a function of $k\geq 2$ variables that satisfies the $i$-margin function $f_i=0$ for all $i$.   For any $(\T,\sigma)\notin \mathfrak{I}(m)$, we have $Q_m(L_nf,(\T,\sigma))=0$.
\end{lem}
\begin{proof}
By the definition of  the reducible graph, we can assume that $(\T,\sigma)$ breaks at $q_0\in T_i$,  and thus $\sigma(q)=q$ for $q\in \text{Range}(T_i)-q_0$.  Thus we have 
\begin{equation*}
\begin{split}
&Q_m(L_nf,(\T,\sigma))\\
=&\int_{(S^d)^{\abs{\T}}}  \text{sgn}(\sigma)f(T_1)\cdots f(T_m)\Pi_{q\in \textrm{Range}(\T)}K_n(x_q,x_{\sigma(q)})d\x\\
=& \text{sgn}(\sigma) \int_{(S^d)^{\abs{\T}-k+1}} \left(\int_{(S^d)^{k-1}}  f(T_i)
\Pi_{q\in \text{Range}(T_i),q\neq q_0}K_n(x_q,x_q)dx_q \right) \\
& \times \left( \Pi_{j\neq i}f(T_j)\right) \left( \Pi_{q'\in \{q_0\} \cup (\text{Range}(\T)-\text{Range}(T_i)) }K_n(x_{q'},x_{\sigma(q')}) \right) dx_{q'}\\
=&0.
\end{split}
\end{equation*}
We have used the assumption $f_i=0$ 	in the last equality.
\end{proof}
Lemma \ref{anobs} implies that 
$$
Q_m(L_nf)=Q_m(L_nf, \mathcal{C}(m))=Q_m(L_nf,  \mathfrak{I}(m)).
$$

Recall the concept of the circle-like graph in Definition \ref{defcircle}, 
we express $\mathfrak{I}(m)$ as the union of 
\begin{equation} 
\begin{split}
E_1:=\{(\T,\sigma)\in \mathfrak{I}(m): (\T,\sigma) \mbox{ is circle-like}\}
\end{split}
\end{equation}
and its complement   \beq E_2:=\mathfrak{I}(m)-E_1.\eeq 

\begin{lem}\label{dm}
For $m\geq 2$,   recall that $a(\sigma)$ is the number of elements that are not  fixed by $\sigma$, we have 
\begin{itemize}
\item  For any $(\T,\sigma)\in \mathfrak{I}(m)$, 
\begin{equation}\label{tsigmak}
km-\abs{\T}+\frac{a(\sigma)}{2} \geq  m.
\end{equation}
\item 
If $(\T,\sigma)\in E_2$ and  $km-\abs{\T}+\frac{a(\sigma)}{2} =  m$,   then $\sigma$ is not a composition of disjoint transpositions, i.e., 
in the cycle decomposition of 
$\sigma$, there must exist at least one cyclic permutation with length  strictly greater than 2. 
\end{itemize}
\end{lem}
\begin{proof} 

We now define two functions $M(i,j)$ and $\Delta(i,j)$  for $1\leq i\leq m,1\leq j\leq k$.  Given a $(\T, \sigma)$-graph, we say an index $q\in[km]$ has \emph{multiplicity $M$} if there are exactly $M$ different $i$'s such that $q\in T_i$.   We define $M(i,j)$ as the multiplicity of $T_{i,j}$. We define $\Delta(i,j)=1$ if $\sigma(T_{i,j})\neq T_{i,j}$ and 0 otherwise. Then we have 
\begin{equation}\label{gm1}
km-\abs{\T}+\frac{a(\sigma)}{2} = \sum_{i=1}^m \sum_{j=1}^k \left( \frac{M(i,j)-1}{M(i,j)}+\frac{\Delta(i,j)}{2M(i,j)}\right).
\end{equation}

Since we assume that $(\T,\sigma)\in \mathfrak{I}(m)$,  for each $i$, $T_i$ has at least two distinct elements,  denoted by  $T_{i, i_1}$ and $T_{i,i_2}$, such that they both have red edges. Therefore, we have
\begin{equation}\label{md}
\max\{M(i,i_1)-1,\Delta(i,i_1) \}\geq 1 \mbox{ and }
\max\{M(i,i_2)-1,\Delta(i,i_2) \}\geq 1.
\end{equation}
If $M(i,i_1)>1$, then 
$$\frac{M(i,i_1)-1}{M(i,i_1)} \geq \frac{1}{2}.$$
If $M(i,i_1)=1$, then by \eqref{md}, $\Delta(i,i_1)=1$. We then have 
$$
\frac{\Delta(i,i_1)}{2M(i,i_1)}=\frac{\Delta(i,i_1)}{2}=\frac{1}{2}.
$$
In both cases we always have 
\begin{equation}\label{mdelta}
\frac{M(i,i_1)-1}{M(i,i_1)}+\frac{\Delta(i,i_1)}{2M(i,i_1)}\geq \frac{1}{2}.
\end{equation}
The same inequality holds for $T_{i,i_2}$. 
Hence we have
\begin{equation}\label{gm-1}
\sum_{j=1}^k \left( \frac{M(i,j)-1}{M(i,j)}+\frac{\Delta(i,j)}{2M(i,j)}\right)\geq 2\times \frac{1}{2}=1. 
\end{equation}
And the equality in \eqref{gm-1} holds iff there are exactly two vertices $(i,i_\alpha), \alpha\in \{1,2\}$ that have red edges and each satisfies
\begin{equation}\label{eqlcase}
M(i,i_\alpha)=1 \mbox{ and } \Delta(i,i_\alpha)=1; \mbox{ or } 	M(i,i_\alpha)=2 \mbox{ and } \Delta(i,i_\alpha)=0.
\end{equation}
By summing over $1\leq i\leq m$, we have 
\begin{equation}\label{gm2}
\sum_{i=1}^m \sum_{j=1}^k \left( \frac{M(i,j)-1}{M(i,j)}+\frac{\Delta(i,j)}{2M(i,j)}\right)\geq
\sum_{i=1}^m 1 =m.
\end{equation}
\eqref{tsigmak} now follows from \eqref{gm1} and \eqref{gm2}.

Now we turn to prove the second part of Lemma \ref{dm} by contradiction.  Suppose  that		$km-\abs{\T}+a(\sigma)/2 = m$ and the cycle decomposition of $\sigma$ only consists of disjoint transpositions, we need  to show $(\T,\sigma)\in E_1$.
By the proof of \eqref{tsigmak} above,
the condition $km-\abs{\T}+a(\sigma)/2 =m$ implies that 
\begin{equation}
\sum_{j=1}^k \left( \frac{M(i,j)-1}{M(i,j)}+\frac{\Delta(i,j)}{2M(i,j)}\right)=1
\end{equation}
for each $1\leq i\leq m$. This further implies that 
for each $1\leq i\leq m$, there are exactly two  vertices $(i,i_1)$ and $(i,i_2)$ that can have red edges, and all the other vertices have no red edges.
By \eqref{eqlcase}, for all $1\leq i\leq m$ and any $\alpha\in\{1,2\}$, either of the following two conditions holds:
\begin{itemize}
\item  $M(i,i_\alpha)=2$ and $\Delta(i,i_\alpha)=0$. In this case 
$(i,i_\alpha)$ has exactly one solid red edge but no red dotted edge.
\item $M(i,i_\alpha)=1$ and $\Delta(i,i_\alpha)=1$. In this case $(i,i_\alpha)$ has at least one dotted red edge, but no solid edge.  Since $\sigma$ is only composed of disjoint transpositions, $(i,i_\alpha)$ must have exactly one dotted red edge connecting with some other vertex $(j, j_{\alpha'})$.  And $j$ has to be distinct from $i$. Otherwise there would be no red edge between the set $\{(i,\cdot)\}$ and 
$\{(i',j'): i'\neq i, 1\leq j'\leq k\}$, which makes $(\T,\sigma)\notin \mathcal{C}(m)$.
\end{itemize}
As a conclusion, in both cases,  for each $1\leq i\leq m$,
there are exactly two vertices in $\{(i,j):1\leq j\leq k\}$ that can have red edge and each of them is connected to vertices in $\{(i',j'):i'\neq i, 1\leq j'\leq k\}$ with a single red edge. This
shows that $(\T,\sigma)$ is circle-like  which is a contradiction, and this proves the second part of Lemma \ref{dm}. 
\end{proof}

The following lemma indicates that  the summation over the subset $E_1$ yields the leading order term of $Q_m(L_nf)$.
\begin{lem}\label{qmcase2}
Fix any $m\geq 2$, we have the following two estimates. 
\begin{enumerate}
\item  	
\begin{equation}\label{qmbde2}
Q_m(L_nf, E_2)=o(n^{(d-1)(k-1)m}).
\end{equation}
\item 	
\begin{equation}\label{qme1bd}
\begin{split}
	Q_m(L_nf,E_1)=&\frac{1}{2}(m-1)! (k(k-1))^m	 \left(\frac{k_n}{s_d}\right)^{mk} \left(\frac{C_{d}}{n^{d-1}}\right)^m\\
	&			\times \int_{(S^d)^{ m}}\hat{h}(x_1,x_2) \hat{h}(x_2,x_3)\cdots \hat{h}(x_m,x_1)dx_1\cdots dx_m\\
	&+  o(n^{(d-1)(k-1)m}), 
\end{split}
\end{equation}
\end{enumerate}
where the constant $C_d$ is defined in Theorem \ref{clt2}, and the symmetric function $\hat h(x,y)$ is defined in \eqref{hath}.

\end{lem}
By the relation $Q_m(L_nf)=Q_m(L_nf, \mathfrak{I}(m))=Q_m(L_nf, E_1)+Q_m(L_nf,E_2)$, we have the following corollary.
\begin{cor}\label{cor:2qm}
For any $m\geq 2$, the $m$-th cumulant satisfies the asymptotic expansion 
\begin{equation}\label{2qme1bd}
\begin{split}
Q_m(L_nf)=&\frac{1}{2}(m-1)! (k(k-1))^m	 \left(\frac{k_n}{s_d}\right)^{mk} \left(\frac{C_{d}}{n^{d-1}}\right)^m\\
&			\times \int_{(S^d)^{ m}}\hat{h}(x_1,x_2) \hat{h}(x_2,x_3)\cdots \hat{h}(x_m,x_1)dx_1\cdots dx_m\\
&+  o(n^{(d-1)(k-1)m}).
\end{split}
\end{equation}In the special case  $m=2$, it yields the limit  \eqref{q2case2} for the variance of $L_nf$. 

\end{cor}
\begin{proof}[Proof of Lemma \ref{qmcase2}]
We first prove part (1).  Given any $(\T,\sigma)\in E_2\subset \mathfrak{I}(m)$, 
by \eqref{tsigmak},  it holds that  $km-\abs{\T}+\frac{a(\sigma)}{2}\geq m$. 
For the case   $km-\abs{\T}+\frac{a(\sigma)}{2}> m$, 
by Lemma \ref{productcontrol}, we have 
\begin{equation}\label{c1}
\begin{split}
&Q_m(L_nf,(\T,\sigma))\\
=&\int_{(S^d)^{\abs{\T}}} f(T_1)\cdots f(T_m)\text{sgn}(\sigma) \Pi_{q\in \text{Range}(\T)} K_n(x_q,x_{\sigma(q)})d\x\\
=&
O(n^{\abs{\T} (d-1)}) O(n^{-(d-1)a(\sigma)/2})\\
=&O(n^{ (d-1)(\abs{\T} -a(\sigma)/2) })
=o(n^{(d-1)(mk-m)}).
\end{split}
\end{equation} 
For the case $km-\abs{\T}+\frac{a(\sigma)}{2}=m$, 
by the second part of Lemma \ref{dm}, 
there must be a cyclic permutation whose length is at least 3 in the cycle decomposition of $\sigma$. Hence by 	Lemma \ref{productcontrol} and Lemma \ref{keylemma}, we can first integrate all variables with indices in that cyclic permutation, and then integrate the remaining variables to get
\begin{equation}\label{c2}
\begin{split}
Q_m(L_nf,(\T,\sigma))				&=
O(n^{\abs{\T} (d-1)})  o(n^{-(d-1)a(\sigma)/2})\\
&=o(n^{ (d-1)(\abs{\T} -a(\sigma)/2) })
=o(n^{(d-1)(mk-m)}).
\end{split}
\end{equation}
By \eqref{c1} and \eqref{c2}, if we sum over all $(\T,\sigma)\in E_2$,  we prove \eqref{qmbde2}.

We next prove part (2). 
We define 
\begin{equation}\label{defhn}
\begin{split}
h_n(x,y):=&\int_{S^d} (f_{1,2}(x,y) -f_{1,2}(x,z))P^2_n(y,z)dz\\
=&
(k_n/s_d)^{-1}  f_{1,2}(x,y) -
\int_{S^d} f_{1,2}(x,z) P_n^2(y,z)dz.\\
\end{split}
\end{equation}
Since $f_{1,2}(x,y)$ and $P_n(x,y)$ depend only on the distance $\d(x,y)$, we have 
\begin{equation}\label{symxy}
\begin{split}
h_n(x,y)=&
(k_n/s_d)^{-1}  f_{1,2}(x,y) -
\int_{S^d} f_{1,2}(x,z) P_n^2(y,z)dz\\
=&(k_n/s_d)^{-1}  f_{1,2}(y,x)-
\int_{S^d} f_{1,2}(y,z) P_n^2(x,z)dz=h_n(y,x).
\end{split}
\end{equation}
Hence $h_n(x,y)$ is symmetric in $x$ and $y$. 
We  claim that 
\begin{equation}\label{qmformula}
\begin{split}
Q_m(L_nf,E_1)=&\frac{1}{2}(m-1)! (k(k-1))^m  \left(\frac{k_n}{s_d}\right)^{mk}
\\	&\times \int_{(S^d)^m}  h_n(x_1,x_2)   h_n(x_2,x_3) \cdots   
h_n(x_m,x_1) dx_1\cdots dx_m.
\end{split}
\end{equation}
Now we prove \eqref{qmformula}. 		Given   $(\T,\sigma)\in E_1$ which is circle-like, by  Proposition \ref{prop:circle-like}, we can find vertices $(i,i_1)$ and $(i,i_2)$ for $1\leq i\leq m$ and a cyclic permutation $p$ of $\{1,\ldots, m\}$ such that
$(i,i_2)$ is connected with $(p(i),p(i)_1)$ with a red edge for all $i$.  To compute $Q_m(L_nf, (\T,\sigma))$, for simplicity, 
by condition  \eqref{sym}  of the permutation invariance of $f$, we assume that  
$i_1=1$ and $i_2=2$ for all $i$, and we also assume $p$ is the cyclic permutation $(12 \cdots m)$. We now define a new kernel $\tilde{P}_i(x,y)$ as follows. If $(i,2)$ and $(i+1,1)$ are connected by a solid red edge (i.e., $T_{i,2}=T_{i+1,1}$),  we let 
$$\tilde{P}_i(x,y)=K_n^{-1}(x,x)\delta_y(x)= (k_n/s_d)^{-1}\delta_y(x),$$  
where $\delta_y(x)$ is a Dirac delta function such that for any function $g$, $$
\int_{S^d} \delta_y(x) g(x)dx=g(y).
$$
If $(i,2)$ and $(i+1,1)$ are connected by a dotted red edge (i.e., $\sigma(T_{i,2})=T_{i+1,1}$ or $\sigma(T_{i+1,1})=T_{i,2}$), then we let $$\tilde{P}_i(x,y)=-P_n^2(x,y)=-(k_n/s_d)^{-2} K_n^2(x,y).$$
Integrating over all variables except those in the set $\{T_{i,\alpha}, 1\leq i\leq m, 1\leq \alpha \leq 2\}$,  
\begin{equation}\label{qmf1}
\begin{split}
Q_m(L_nf,(\T,\sigma))
=&\left(\frac{k_n}{s_d}\right)^{mk} \int_{(S^d)^{2m}}f_{1,2}(x_1,y_1)\tilde{P}_1(y_1,x_2)
f_{1,2}(x_2,y_2)\tilde{P}_2(y_2,x_3) \times \\
&\cdots \times  f_{1,2} (x_m,y_m)\tilde{P}_m(y_m,x_1)   
dx_1\cdots dx_m dy_1\cdots dy_m .
\end{split}
\end{equation}
If we fix the cyclic permutation $p=(1\cdots m)$ and indices $i_\alpha=1,2$, then we can get $2^m$ different $(\T,\sigma)$ in the set $E_1$, because each red edge between $(i,2)$ and $(i+1,1)$ can either be a solid one, or a dotted one.  If we sum over all $2^m$ different $(\T,\sigma)$ in \eqref{qmf1} and integrate over the variables $y_1,\ldots, y_m$, then we get a total contribution of 
\begin{equation}\label{mhn}
\begin{split}
\left(\frac{k_n}{s_d}\right)^{mk} 	\int_{(S^d)^m}  h_n(x_1,x_2)   h_n(x_2,x_3) \times \cdots   
\times h_n(x_m,x_1) dx_1\cdots dx_m.
\end{split}
\end{equation}

Since there are $(m-1)!$ cyclic permutations of $[m]$ and there are $(k(k-1))^m$ distinct combinations of the indices $i_\alpha,1\leq i\leq m, 1\leq \alpha\leq 2$, we obtain \eqref{qmformula}. But note that there is a  factor $1/2$ in the front of  \eqref{qmformula}, this is because given a circle-like $(\T,\sigma)$-graph,  the correspondence from $p$ and  $\{i_\alpha, 1\leq i\leq m, \alpha=1,2\}$ to $(\T,\sigma)$ is not 1-1, but rather 2-1. Indeed, by defining $p'=p^{-1}$ and $i'_\alpha=i_{3-\alpha}$, we end up at the same $(\T,\sigma)$-graph.  As an example,  the $(\T,\sigma)$-graph given in the left panel of Figure \ref{Example2} is circle-like,  and by Proposition \ref{prop:circle-like} we can take the cyclic permutation as $p=(123)$ or $p=(132)$.

Now we prove part (2) of Lemma \ref{qmcase2}. 
By  \eqref{rmk:lem2},  for any fixed $x$ and $y$, 
\begin{equation}\label{h}
\begin{split}
\lim_{n\to\infty} n^{d-1}h_n(x,y)&=C_{d}
\int_{S^d}(f_{1,2}(x,y)-f_{1,2}(x,z))\sin^{-(d-1)}(\arccos(z\cdot y))dz\\
&=			C_{d}\hat{h}(x,y).
\end{split}
\end{equation}
Furthermore, by the boundedness of $f$ and \eqref{reproduce4}, there exists  constants $c$ and $C$ such that for all $x$ and $y$, we have 
\begin{equation}\label{unifbo}
\abs{n^{d-1}h_n(x,y)}\leq c n^{d-1}\int_{S^d}P_n^2(y,z)dz \leq C. 
\end{equation}
Hence, part (2) of Lemma \ref{qmcase2} follows from \eqref{qmformula}, \eqref{h} and the dominated convergence theorem. 
\end{proof}
\subsection{Identification of the limiting distribution}
Recall from \eqref{symxy} that $h_n$ and thus  $\hat{h}$ are both symmetric, i.e., $\hat{h}(x,y)=\hat{h}(y,x)$. This implies that there exists an orthonormal basis of $L^2(S^d)$, say $w_j,j\geq 1$ such that
$$\hat{h}(x,y)=\sum_{j=1}^{\infty}
z_jw_j(x)w_j(y)$$
for almost all $(x,y)\in S^d \times S^d$. 

We consider the following random variable $$
X_n:=\Big(L_nf -\E(L_nf)\Big)\left(\frac{k_n}{s_d} \right)^{-k} \left(\frac{C_{d} k(k-1)}{n^{d-1}}\right)^{-1}.
$$
By Corollary \ref{cor:2qm},  for any fixed $m\geq 2$,  we have 
\begin{equation}\label{qmxnlimit}
\lim_{n\to\infty}Q_m(X_n)=\frac{(m-1)!}{2}\sum_{j=1}^{\infty} z_j^m. 
\end{equation}
In addition, $Q_1(X_n)=\E(X_n)=0$ for all $n$. 

We shall now determine the specific form of the limiting distribution of $X_n$ in three steps. Let $\chi_i,i\geq 1$ be independent  chi-squared random variables with one degree of freedom, defined on some common probability space $\Omega_0$. We consider a sequence of random variables
$Y_N$ defined by
$$
Y_N= \sum_{i=1}^N z_i (\chi_i-1)/2.
$$
\begin{itemize}
\item We show that $Y_N, N\geq 1$ is a Cauchy sequence in $L^2(\Omega_0)$. Thus  $Y_N$ converges to some limiting random variable $Y$ in the $L^2$ norm. We further show that the convergence is also in $L^m$ for any $m\geq 1$, which implies that   $Q_m(Y_N)\to Q_m(Y)$ for any $m\geq 1$. 
\item We next find the cumulants of $Y$ by computing $\log \E\exp(\mathrm itY_N)$ and taking the limit $N\to \infty$. It turns out that 
$$
Q_m(Y)=\lim_{n\to\infty}Q_m(X_n), \, \forall m\geq 1.
$$
\item Finally, we prove that the distribution of $Y$ satisfies the Carleman's condition. This combined with the second step shows that $X_n$ converges to $Y$ in distribution and completes the proof of Theorem \ref{clt2}. 
\end{itemize}

By \eqref{h} and \eqref{unifbo}, $\hat{h}(x,y)$ is uniformly bounded,  and thus  we have $$ {\sum_{j=1}^{\infty} z_j^2}=\int_{S^d}\int_{S^d}\hat{h}(x,y)^2dxdy<\infty.$$ 
Thus for any $N_1<N_2$, it holds that 
$$
\norm{Y_{N_1}-Y_{N_2}}_{L^2}^2\leq C\sum_{j=N_1+1}^{N_2} z_j^2 \to 0\,\,  \mbox{ as }\,\,N_1, N_2\to \infty,
$$
which implies that $Y_N,N\geq 1$ is a Cauchy sequence in $L^2(\Omega_0)$.
Consequently, we can find a limiting random variable $Y$ such that $Y_N\to Y$ in $L^2$.  

For all $m\geq 2$,   one has 
\begin{equation}\label{sumzm}
\sum_{j=1}^{\infty} \abs{z_j}^m\leq \Big(\sum_{j=1}^\infty z_j^2\Big)^{m/2}
<\infty.
\end{equation}
By \eqref{sumzm}, for any even integer $m$, the $m$-th moment of $Y_{N}$ is  bounded uniformly from above: 
$$
\E(Y_N^m)\leq  C\sum_{m=m_1+\cdots+m_{\ell}: m_1,\ldots, m_{\ell}\geq 2}
\prod_{i=1}^{\ell}  	 \Big(\sum_{j=1}^{\infty} \abs{z_j}^{m_i}\Big)<\infty,
$$
where the summation  is  over all integer partitions of $m$. 
This further implies  the sequence $\{Y_N^m,N\geq 1\}$
is uniformly integrable for any fixed $m\geq 1$ and thus we have 
$$
Y_N \xrightarrow{L^m} Y \,\,\mbox{ as }\,\,N\to\infty
$$  
for all $m\geq 1$.  We can formally write $Y$ as the sum $\sum_{i=1}^\infty z_i(\chi_i-1)/2$.

We now turn to the second step. 
The cumulant generating function of $Y_N$ is 
\begin{equation*}
\begin{split}
\log \E\exp(\i tY_N)=&\sum_{i=1}^N  \log \E\exp( z_i (\chi_i-1) \i t/2)\\
=&\sum_{i=1}^N
\log \frac{1}{\sqrt{1-z_i \i t}}-\sum_{i=1}^N \frac{ z_i\i t}{2} =\sum_{i=1}^N \sum_{m=2}^{\infty} \frac{z_i^m}{2m}(\i t)^m.
\end{split}
\end{equation*}
For $m\geq 2$, the $m$-th cumulant of $Y_N$ is 
\begin{equation}\label{qmx}
Q_m(Y_N)=m!  \sum_{j=1}^{N} \frac{z_j^m}{2m}=\frac{(m-1)!}{2}\sum_{i=1}^N z_i^m.
\end{equation}
Since $Y_N$ converges to $Y$ in  $L^m$ for all $m\geq 1$,  by \eqref{cummu_moment} we have 
\begin{equation}\label{qmy}
Q_m(Y)= \lim_{N\to\infty}Q_m(Y_N)=\frac{(m-1)!}{2}\sum_{i=1}^{\infty} z_i^m,\, \forall\, m\geq 2, 
\end{equation}
which coincides with \eqref{qmxnlimit}. Also,  $Q_1(Y)=\E(Y)=0$ since $\E(Y_N)=0$ for all $N$. 

To finish the proof of the convergence of $X_n$ to $Y$,  we need to show that the distribution of  $Y$ is uniquely determined by the   cumulant condition \eqref{qmy}. To this end it suffices to verify the  Carleman's condition
\begin{equation}\label{carl}
\sum_{m=1}^{\infty} \left(\E(Y^{2m})\right)^{-1/(2m)}=\infty.
\end{equation}
To establish \eqref{carl},  by \eqref{sumzm} and \eqref{qmy}, for $m\geq 2$, we have
\begin{equation}\label{qmxbd}
\abs{Q_m(Y)}\leq m! C^m,
\end{equation}
for some constant $C>0$. 
Note that \eqref{qmxbd} also holds for $m=1$ since $Q_1(Y)=0$.  By \eqref{moment_cummu} and \eqref{qmxbd}, for any even integer $m$, we have 
\begin{equation}
\begin{split}\E(Y^m) \leq &	\sum_{ R=\{R_1, \ldots, R_{\ell}\} \in \Pi(m)   }\abs{Q_{\abs{R_1}}}\cdots \abs{Q_{\abs{R_{\ell}}}}\\
\leq &\sum_{ R=\{R_1, \ldots, R_{\ell}\} \in \Pi(m)   } \abs{R_1}! C^{\abs{R_1}}\cdots \abs{R_{\ell}}!C^{\abs{R_{\ell}}}\\
=& C^m \sum_{ R=\{R_1, \ldots, R_{\ell}\} \in \Pi(m)   } \abs{R_1}! \cdots \abs{R_{\ell}}!,
\end{split}
\end{equation}
where we used the fact that $\sum_{i=1}^{\ell}\abs{R_i}=m$. 
To estimate the last summation, given an integer partition $m=m_1+\cdots+m_{\ell}$ for some   $\ell\geq 1$ and $m_1,\ldots, m_{\ell}\geq 1$. Denote the number of distinct $m_i$'s by $N'$ and let $\tau_1,\ldots, \tau_{\ell}$ be their multiplicities. 
Then  the number of partitions $\{R_1,\ldots, R_{\ell}\}$ of $[m]$ such that $$\# \{R_i: \abs{R_i}=m_i\}=\tau_i, \quad  \forall \, 1\leq i\leq \ell$$ is given by
$$
\frac{m!}{m_1!\cdots m_\ell!  \prod_i ^{N'}\tau_i! }.
$$  
Thus, we have 
\begin{equation}\label{bdem1}
\begin{split}
\E(Y^m) \leq & C^m \sum_{m=m_1+\cdots+m_{\ell}} 
\sum_{R\in \Pi(m), \abs{R_i}=m_i}m_1!\cdots m_{\ell}!\\
\leq &  C^m \sum_{m=m_1+\cdots+m_{\ell}}  
\frac{m!}{m_1! \cdots m_{\ell}! \prod_i ^{N'}\tau_i! } m_1! \cdots m_{\ell}!\\
\leq & C^m  m! \sum_{m=m_1+\cdots+m_{\ell}}  1.
\end{split}
\end{equation}
It is known  that the total number of partitions of an integer $m$, denoted by $\kappa(m)$, satisfies $\log \kappa (m)\sim \pi \sqrt{(2m)/3}$ as $m\to\infty$ (p.70 in  \cite{A}). 
Consequently,  for some constant $\tilde{C}$ large enough and all $m\geq 1$, $\kappa (m)\leq \tilde{C}^m $. Therefore, we have 
\begin{equation}\label{bdem2}
\E(Y^m) \leq C^m m! \tilde{C}^m \leq (C\tilde{C})^m m^m.
\end{equation}
Now \eqref{carl} follows from \eqref{bdem2} since
\begin{equation}
\sum_{i=1}^{\infty} \left(\E(Y^{2i})\right)^{-1/(2i)} \geq 
\sum_{i=1}^{\infty} \left( (C\tilde{C})^{2i} (2i)^{2i}\right)^{-1/(2i)}\geq \sum_{i=1}^{\infty}\frac{c}{i}=\infty. 
\end{equation}
This completes the proof of \eqref{carl},  and thus we finish the proof of Theorem \ref{clt2}.


\begin{thebibliography}{99}


\bibitem{A}
George~E. Andrews.
\newblock {\em The theory of partitions}.
\newblock Cambridge Mathematical Library. Cambridge University Press,
Cambridge, 1998.
\newblock Reprint of the 1976 original.

\bibitem{AH}
Kendall Atkinson and Weimin Han.
\newblock {\em Spherical harmonics and approximations on the unit sphere: an
introduction}, volume 2044 of {\em Lecture Notes in Mathematics}.
\newblock Springer, Heidelberg, 2012.



\bibitem{BYY}
B.~B{\l}aszczyszyn, D.~Yogeshwaran, and J.~E. Yukich.
\newblock Limit theory for geometric statistics of point processes having fast
decay of correlations.
\newblock {\em Ann. Probab.}, 47(2):835--895, 2019.

\bibitem{BO}
Omer Bobrowski and Goncalo Oliveira.
\newblock Random \v{C}ech complexes on {R}iemannian manifolds.
\newblock {\em Random Structures Algorithms}, 54(3):373--412, 2019.


\bibitem{JS}
Svante Janson.
\newblock {\em Gaussian {H}ilbert spaces}, volume 129 of {\em Cambridge Tracts
in Mathematics}.
\newblock Cambridge University Press, Cambridge, 1997.

\bibitem{KM}
Matthew Kahle and Elizabeth Meckes.
\newblock Limit theorems for {B}etti numbers of random simplicial complexes.
\newblock {\em Homology Homotopy Appl.}, 15(1):343--374, 2013.


\bibitem{TY}
Tomoyuki Shirai and Yoichiro Takahashi.
\newblock Random point fields associated with certain {F}redholm determinants.
{I}. {F}ermion, {P}oisson and boson point processes.
\newblock {\em J. Funct. Anal.}, 205(2):414--463, 2003.


\bibitem{AS2}
Alexander Soshnikov.
\newblock The central limit theorem for local linear statistics in classical
compact groups and related combinatorial identities.
\newblock {\em Ann. Probab.}, 28(3):1353--1370, 2000.

\bibitem{AS}
Alexander Soshnikov.
\newblock Gaussian limit for determinantal random point fields.
\newblock {\em Ann. Probab.}, 30(1):171--187, 2002.

\bibitem{So}
Gabor Szeg\H o.
\newblock {\em Orthogonal {P}olynomials}.
\newblock American Mathematical Society Colloquium Publications, Vol. 23.
American Mathematical Society, New York, 1939.

\bibitem{YA}
D.~Yogeshwaran and Robert~J. Adler.
\newblock On the topology of random complexes built over stationary point
processes.
\newblock {\em Ann. Appl. Probab.}, 25(6):3338--3380, 2015.




\end{thebibliography}
\end{document}